\documentclass[11pt,reqno]{amsart}

\usepackage[T1]{fontenc}
\usepackage[a4paper,margin=29mm]{geometry}
\usepackage{lmodern}
\usepackage{microtype}
\microtypesetup{expansion=false}
\usepackage{amsmath,amssymb,amsthm}
\usepackage{mathtools}
\usepackage{mathrsfs}
\usepackage{enumitem}
\usepackage{tikz-cd}
\usepackage{xcolor}
\usepackage{aliascnt}
\usepackage{hyperref}
\usepackage[nameinlink,capitalize,noabbrev]{cleveref}

\hypersetup{
  colorlinks=true,
  linkcolor=blue!55!black,
  citecolor=green!40!black,
  urlcolor=blue!65!black
}

\numberwithin{equation}{section}

\newtheorem{theorem}{Theorem}[section]

\newaliascnt{proposition}{theorem}
\newtheorem{proposition}[proposition]{Proposition}
\aliascntresetthe{proposition}

\newaliascnt{lemma}{theorem}
\newtheorem{lemma}[lemma]{Lemma}
\aliascntresetthe{lemma}

\newaliascnt{corollary}{theorem}
\newtheorem{corollary}[corollary]{Corollary}
\aliascntresetthe{corollary}

\newaliascnt{assumption}{theorem}
\newtheorem{assumption}[assumption]{Assumption}
\aliascntresetthe{assumption}

\theoremstyle{definition}

\newaliascnt{definition}{theorem}
\newtheorem{definition}[definition]{Definition}
\aliascntresetthe{definition}

\newaliascnt{example}{theorem}

\aliascntresetthe{example}

\theoremstyle{remark}

\newaliascnt{remark}{theorem}
\newtheorem{remark}[remark]{Remark}
\aliascntresetthe{remark}

\crefname{proposition}{Proposition}{Propositions}
\Crefname{proposition}{Proposition}{Propositions}
\crefname{lemma}{Lemma}{Lemmas}
\Crefname{lemma}{Lemma}{Lemmas}
\crefname{corollary}{Corollary}{Corollaries}
\Crefname{corollary}{Corollary}{Corollaries}
\crefname{assumption}{Assumption}{Assumptions}
\Crefname{assumption}{Assumption}{Assumptions}
\crefname{definition}{Definition}{Definitions}
\Crefname{definition}{Definition}{Definitions}
\crefname{example}{Example}{Examples}
\Crefname{example}{Example}{Examples}
\crefname{remark}{Remark}{Remarks}
\Crefname{remark}{Remark}{Remarks}

\DeclareMathOperator{\divisor}{div}

\DeclareMathOperator{\Gr}{Gr}

\DeclareMathOperator{\Spec}{Spec}

\newcommand{\A}{\mathbb A}

\newcommand{\Pj}{\mathbb P}

\newcommand{\bH}{\mathbb H}
\newcommand{\DR}{\operatorname{DR}}

\newcommand{\StacksTag}[1]{\href{https://stacks.math.columbia.edu/tag/#1}{Tag~#1}}
 
\title[Irregular Hodge Bundles with Fixed Pole Orders]
{Irregular Hodge Bundles for Deligne--Mumford\\
Landau--Ginzburg Families with Fixed Pole Orders}
\author{Haoxu Wang}
\address{Morningside Center of Mathematics, Academy of Mathematics and
Systems Science, Chinese Academy of Sciences, Beijing 100190, China}
\email{krassotkinkolya@gmail.com}
\hypersetup{
  pdftitle={Irregular Hodge Bundles for Deligne--Mumford
    Landau--Ginzburg Families with Fixed Pole Orders},
  pdfauthor={Haoxu Wang},
  pdfsubject={Irregular Hodge theory in families},
  pdfkeywords={irregular Hodge filtration, Deligne--Mumford stacks,
    Landau--Ginzburg families, Katz--Oda connection, coefficient
    invariance, toric stacks}
}
\date{}

\begin{document}
\begin{abstract}
We prove that the canonical irregular Hodge filtrations in smooth
families of Deligne--Mumford Landau--Ginzburg models with fixed pole
orders assemble into filtered algebraic vector bundles.  More
precisely, twisted de Rham cohomology forms a vector bundle carrying
a canonical algebraic integrable connection, every rational
irregular Hodge level is a subbundle with locally free quotient, and
every rational graded piece is a vector bundle.  These constructions
commute with arbitrary finite-type base change, and the connection
satisfies shifted Griffiths transversality.  Consequently, all
filtered and graded dimensions are locally constant.

As an application, for a smooth quasiprojective toric
Deligne--Mumford stack and every fixed Newton polytope at infinity
satisfying the support condition, these conclusions hold over the
full coefficient locus of Laurent polynomials that are nondegenerate
at infinity.  In particular, their filtered and graded dimensions
are independent of the coefficients.  When the resulting stacky
fans form a Clarke dual pair, combining coefficient invariance with
Harder--Lee's combinatorial formula shows that the ordinary and
orbifold irregular Hodge numbers of every potential in the full
nondegenerate coefficient locus are computed by their combinatorial
\(\Xi\)-complex.
\end{abstract}

\maketitle
\setcounter{tocdepth}{1}
\tableofcontents

\section{Introduction}
\label{sec:introduction}

\subsection{The relative irregular Hodge problem}
\label{subsec:introduction-relative-problem}

For a smooth proper family, de Rham cohomology forms a vector bundle
with the Gauss--Manin connection, while the Hodge filtration is
realized by subbundles satisfying Griffiths transversality.  The aim
of this paper is to establish an irregular analogue for smooth
families of Deligne--Mumford Landau--Ginzburg models.

Let \(\mathscr U\) be a smooth separated complex Deligne--Mumford
stack and let
\[
 w\colon\mathscr U\longrightarrow\A^1
\]
be a regular function.  The exponentially twisted differential
\(d+dw\wedge\) defines
\[
 H_{\mathrm{dR}}^k(\mathscr U,w)
 :=
 \bH^k\!\left(
  \mathscr U,
  (\Omega_{\mathscr U}^\bullet,d+dw\wedge)
 \right).
\]
For smooth quasiprojective varieties, Yu introduced a decreasing
rationally indexed irregular Hodge filtration on these groups, and
Esnault--Sabbah--Yu proved the corresponding \(E_1\)-degeneration
\cite{Yu,ESY}.  Sabbah--Yu and Sabbah subsequently placed the
construction in the theory of exponentially twisted mixed Hodge
modules and irregular Hodge modules
\cite{SabbahYu,SabbahIHT}.

Stacky and orbifold versions of the construction have been developed
in the toric setting by Harder--Lee, while compactification
independence and \(E_1\)-degeneration for general Deligne--Mumford
Landau--Ginzburg models are established in the two preceding papers
in this series
\cite{HarderLee,WangCompactification,WangEOne}.

Harder--Lee also study variation in a distinguished one-parameter
toric family associated with a tropical weight
\cite[\S\S8--9]{HarderLee}.  Over the punctured parameter disc they
construct the logarithmic Gauss--Manin connection and establish the
filtered comparison needed for coefficient invariance; their
extension across the origin is a quasi-stable degeneration whose
central fibre is analyzed by tropical methods.

The present paper develops the smooth fixed-pole part of this picture
as a general relative algebraic theory.  For families with fixed pole
orders over an arbitrary smooth finite-type base, we prove that
twisted de Rham cohomology, the canonical irregular Hodge levels, and
their rational graded pieces form vector bundles compatible with base
change.  The total cohomology carries an integrable connection
satisfying shifted Griffiths transversality.  Unlike the degeneration
theory of Harder--Lee, the results here concern the smooth parameter
locus and make no assertion about a degenerate central fibre.

\subsection{Irregular Hodge bundles in families with fixed pole
orders}
\label{subsec:introduction-fixed-pole-bundles}

Let \(S\) be a smooth connected finite-type complex scheme.  A smooth
compactified Deligne--Mumford Landau--Ginzburg family with fixed pole
orders consists of a proper smooth morphism
\[
 \pi\colon\mathfrak X\longrightarrow S
\]
from a tame Deligne--Mumford stack, a relative labelled SNC divisor
\(\mathfrak D=\bigcup_i\mathfrak D_i\), a prescribed relative polar
divisor
\[
 \mathfrak P=\sum_i e_i\mathfrak D_i
\]
whose integers \(e_i\geq0\) are independent of the point of the base,
and a rational function \(\mathcal W\) that is regular on
\(\mathfrak U:=\mathfrak X\setminus\mathfrak D\).  The polar
divisor of \(\mathcal W_s\) is \(\mathfrak P_s\) on every geometric
fibre, and each
\((\mathfrak X_s,\mathfrak D_s,\mathcal W_s)\) is an NC rational
compactification with projective coarse space.  Both the compactified
fibre and the potential may vary with \(s\), while the labelled
boundary and the polar multiplicities are fixed over the base.  The
precise definition is given in
\cref{ass:fixed-pole-orders-dm-family}.

Our main geometric result is the following.

\begin{theorem}[Irregular Hodge bundles in a family with fixed pole
orders]
\label{thm:introduction-fixed-pole-dm-family}
\label{thm:introduction-fixed-pole-orders}
Let
\((\pi\colon\mathfrak X\to S,\mathfrak D,\mathfrak P,
\mathcal W)\) be a smooth family with fixed pole orders, and put
\(\mathfrak U=\mathfrak X\setminus\mathfrak D\).  Then:
\begin{enumerate}[label=\textup{(\roman*)}]
\item For every \(\lambda\in\mathbb Q\), the relative Yu lattice
  has perfect direct image and commutes with arbitrary finite-type
  derived base change.  At a geometric point \(s\to S\), its derived
  fibre is the derived global section complex of the fibrewise Yu
  level
  \[
   F^\lambda K_{\mathfrak X_s,\mathcal W_s}^\bullet.
  \]

\item For every \(q\), the groups
  \(H_{\mathrm{dR}}^q(\mathfrak U_s,\mathcal W_s)\) are the fibres of
  a finite locally free sheaf \(\mathcal H_S^q\) carrying a canonical
  algebraic integrable connection
  \[
   \nabla_S^q\colon
   \mathcal H_S^q
   \longrightarrow
   \Omega_{S/\mathbb C}^1\otimes\mathcal H_S^q.
  \]
  Formation of \(\mathcal H_S^q\) commutes with arbitrary finite-type
  base change.

\item For every \(q\) and \(\lambda\), the canonical level
  \[
   F_{\mathrm{irr}}^\lambda
   H_{\mathrm{dR}}^q(\mathfrak U_s,\mathcal W_s)
  \]
  is the fibre of a subbundle
  \(\mathcal I_\lambda^q\subseteq\mathcal H_S^q\) with locally free
  quotient.  Every rational graded piece
  \(\Gr_{F_{\mathrm{irr}}}^\lambda H_{\mathrm{dR}}^q\) is likewise
  the fibre of a vector bundle.  All these bundles and their
  transition maps commute with arbitrary finite-type base change.

\item The connection satisfies shifted Griffiths transversality
  \[
   \nabla_S^q(\mathcal I_\lambda^q)
   \subseteq
   \Omega_{S/\mathbb C}^1\otimes
   \mathcal I_{\lambda-1}^q.
  \]
  It is functorial under smooth base change.  Consequently, all
  total, filtered, and rational graded dimensions, and hence all
  irregular Hodge numbers, are independent of \(s\in S\).
\end{enumerate}
\end{theorem}

\subsection{Proof strategy and scope}
\label{subsec:introduction-strategy-scope}

The exponentially twisted de Rham complex introduced in
\cref{subsec:introduction-relative-problem} has the Yu-filtered
logarithmic realization on \(\mathfrak X\)
\[
 \mathcal A_S^\bullet
 :=
 \left(
  \Omega_{\mathfrak X/\mathbb C}^\bullet
  (\log\mathfrak D)(*\mathfrak P),
  d+d\mathcal W\wedge
 \right)
\]
with its Yu filtration \(F\); write
\(\mathcal A_S^{\lambda,\bullet}:=F^\lambda
\mathcal A_S^\bullet\).  Exterior multiplication by absolute forms
makes this filtered complex a filtered absolute de Rham dg module.
We call it the \emph{absolute logarithmic twisted de Rham dg module}
to distinguish it from the relative logarithmic twisted de Rham dg
module
\[
 \mathcal C_S^\bullet
 :=
 \left(
  \Omega_{\mathfrak X/S}^\bullet
  (\log\mathfrak D)(*\mathfrak P),
  d_{\mathfrak X/S}+d_{\mathfrak X/S}\mathcal W\wedge
 \right).
\]
The action of forms pulled back from \(S\) defines the Katz--Oda
filtration \(L\) on \(\mathcal A_S^\bullet\), and
\(\mathcal C_S^\bullet=\operatorname{gr}_L^0
\mathcal A_S^\bullet\).  Its induced Yu filtration is given by
\(\mathcal C_S^{\lambda,\bullet}:=\operatorname{gr}_L^0
\mathcal A_S^{\lambda,\bullet}\).  More generally, the logarithmic
cotangent sequence gives canonical identifications
\begin{equation}
 \operatorname{gr}_L^r
 \mathcal A_S^{\lambda,\bullet}
 \simeq
 \pi^{-1}\Omega_{S/\mathbb C}^r
 \otimes
 \mathcal C_S^{\lambda-r,\bullet}[-r].
\label{eq:introduction-base-degree-graded-yu}
\end{equation}

The proof exploits the two structures carried simultaneously by
\(\mathcal A_S^\bullet\).  Its absolute de Rham dg-module structure,
together with the filtration \(L\), is the
\(D\)-module-inspired part of the argument; see
\cref{rem:abstract-filtered-d-module-origin}.  The formalism developed
in \cref{sec:de-rham-dg-modules-katz-oda} extracts the shifted
Katz--Oda operators from
\eqref{eq:introduction-base-degree-graded-yu} and constructs the
integrable connection on the limiting direct image in
\cref{prop:abstract-shifted-katz-oda}.  On the other hand, the
relative levels \(\mathcal C_S^{\lambda,\bullet}\) are bounded
\(S\)-linear complexes with coherent \(S\)-flat terms.
\Cref{sec:perfect-s-linear-complexes} supplies their perfect direct
images, derived base change, and the relevant perfect-complex fibre
criteria.  In particular,
\cref{lem:part-iii-constant-perfect-fibre-cohomology} converts locally
constant derived-fibre dimensions into locally free cohomology
sheaves with cohomology base change.

These two inputs meet in
\cref{thm:katz-oda-perfectness-exactness}, which is the central
mechanism of the paper.  The integrable connection on the limiting
direct image first makes the total cohomology into a vector bundle.
Fibrewise \(E_1\)-degeneration then controls how the filtration levels
sit inside this bundle.  Combined with the perfect-complex fibre
criteria, this promotes the fibrewise filtration to a filtration by
subbundles whose graded pieces are again vector bundles and whose
formation commutes with base change.  The Katz--Oda construction also
gives the resulting filtration its shifted Griffiths transversality.

In \cref{sec:families-fixed-pole-orders} we show that
\((\mathcal A_S^\bullet, F)\) satisfies the hypotheses of this
mechanism.  The companion results recalled in
\cref{sec:compactification-comparison-degeneration} supply the
required fibrewise exactness and identify the induced filtration on
each fibre with the canonical irregular Hodge filtration.
\Cref{thm:katz-oda-perfectness-exactness} then gives the bundles, base-change
properties, connection, and transversality asserted in
\cref{thm:introduction-fixed-pole-orders}.

This formulation retains the de Rham mechanism of filtered
\(D\)-module theory without invoking a category of irregular filtered
\(D\)-modules on Deligne--Mumford stacks or a stacky
filtered-direct-image strictness theorem.  For a smooth variety
admitting a projective good compactification, the strict filtered
direct-image theorem of Sabbah--Yu gives an alternative \(D\)-module
route and explains the close relation with Kontsevich bundles and
rescaling \cite{SabbahYu,SabbahIHT}; see
\cref{rem:part-iii-d-module-interpretation-filtered-gap}.

\subsection{Relation with previous deformation results}
\label{subsec:introduction-previous-deformation-results}

The one-parameter construction in
\cite[\S\S8--9]{HarderLee} is a direct predecessor of the relative
mechanism used here.  Harder--Lee consider a distinguished family of
Laurent potentials of the form
\[
 1+\sum_m u_m t^{\varphi(m)}x^m
\]
over a punctured disc.  On its smooth locus the boundary and polar
multiplicities remain fixed, and their logarithmic Gauss--Manin
construction and shifted transversality are closely related to the
Katz--Oda construction below.  They then extend the family over the
origin and prove the required filtered comparison for the resulting
quasi-stable toric degeneration by a tropical analysis of the central
fibre.

The present result differs in two directions.  It treats smooth
fixed-pole families over arbitrary algebraic bases, including
Deligne--Mumford families in which both the compactified space and the
potential may vary, and it constructs the filtered and graded bundles
themselves together with arbitrary base change, exactness, connection,
and transversality.  In the toric application it applies over the
entire nondegenerate coefficient locus, rather than only along a
distinguished one-parameter family.  Conversely, it does not treat
the degenerate parameter value or replace the central-fibre argument
of Harder--Lee.

The fibrewise injectivity used in our general Deligne--Mumford setting
is supplied by the stack \(E_1\)-degeneration theorem
\cite{WangEOne}.  This should be distinguished from the special
central-fibre injectivity proved by tropical methods in
\cite[\S9]{HarderLee}.

Numerical deformation invariance also has other important
predecessors.  In particular, Qin--Zhang prove invariance of irregular
Hodge numbers for nondegenerate functions on a fixed smooth projective
SNC pair \cite{QinZhang}.  The result here has a different output: it
constructs the relative filtered and graded bundles themselves,
together with their exact sequences, base-change properties,
connection, and transversality, and it includes the stacky and
orbifold settings.  These relative structures, rather than numerical
constancy alone, are the content of
\cref{thm:introduction-fixed-pole-orders}.

\subsection{Toric and orbifold applications}
\label{subsec:introduction-toric-orbifold-applications}

The condition that the pole orders remain fixed occurs naturally for
coefficient families of Laurent potentials.  Let \(\mathscr U_\Gamma\)
be a smooth
quasiprojective toric Deligne--Mumford stack; torus factors and a
finite generic stabilizer are allowed.  Fix a lattice polytope
\(P\) in the character space, containing the origin and compatible
with the open toric fan, and let
\[
 V_P:=\bigoplus_{m\in P\cap M}\mathbb Cx^m
\]
be the corresponding coefficient space.  For
\(h=\sum_ma_mx^m\), write
\[
 \Delta_\infty(h)
 :=\operatorname{Conv}
 \bigl(\{0\}\cup\{m\mid a_m\ne0\}\bigr)
\]
for its Newton polytope at infinity.  Let
\(\mathcal U_P\subseteq V_P\) be the locus on which
\(\Delta_\infty(h)=P\) and \(h\) is Newton nondegenerate at infinity.
\begin{theorem}[Toric and orbifold coefficient invariance]
\label{thm:introduction-coefficient-invariance}
The locus \(\mathcal U_P\) is a nonempty Zariski-open irreducible
subset of \(V_P\).  The universal Laurent potential and its
restrictions to the inertia sectors, before and after effective
rigidification, admit coefficient-independent smooth toric
compactifications that form families with fixed pole orders over
\(\mathcal U_P\).
Consequently, the conclusions of
\cref{thm:introduction-fixed-pole-orders}, together with their
sectorwise age-shifted orbifold analogues, hold over
\(\mathcal U_P\).  In particular, the canonical irregular Hodge
filtrations and rational graded pieces form vector bundles compatible
with arbitrary base change, and all ordinary and orbifold irregular
Hodge numbers are independent of \(h\in\mathcal U_P\).
\end{theorem}

The family selected by Harder--Lee from a tropical weight determines
a distinguished analytic curve in \(\mathcal U_P\).  Their
degeneration argument computes the irregular Hodge numbers along this
punctured family and relates them to the tropical central fibre.  The
preceding theorem promotes the resulting value from that family to
every point of the irreducible algebraic coefficient space
\(\mathcal U_P\), and at the same time supplies the filtered and
graded bundle structures over that space.

Combined with Harder--Lee's combinatorial description, this has the
following consequence.  Under the freeness, full-dimensionality, and
convexity hypotheses of
\cref{cor:toric-harder-lee-all-nondegenerate-coefficients}, a coherent
star triangulation of \(P\) produces a Clarke dual pair
\((\boldsymbol\Sigma,\check{\boldsymbol\Sigma})\) with
\(T(\boldsymbol\Sigma)\simeq\mathscr U_\Gamma\).  For every
\(h\in\mathcal U_P\), its ordinary and orbifold irregular Hodge
numbers are respectively the dimensions of the identity-sector and
full combinatorial Hodge spaces defined by Harder--Lee's cellular
sheaves \(\Xi\) on
\((\boldsymbol\Sigma\oplus\check{\boldsymbol\Sigma})_0\).
Thus their formula holds on the entire nondegenerate coefficient
locus \cite{HarderLee}.
 \section{Compactification Comparison and
\texorpdfstring{\(E_1\)}{E1}-Degeneration}
\label{sec:compactification-comparison-degeneration}

We use two preceding papers in this series:
\emph{Compactification Independence of the Irregular Hodge
Filtration on Deligne--Mumford Stacks}
\cite{WangCompactification} and
\emph{\(E_1\)-Degeneration for Irregular Hodge Filtrations on
Deligne--Mumford Stacks} \cite{WangEOne}.  We refer to them as the
first and second companion papers.

Throughout, stacks are separated and of finite type over \(\mathbb C\),
and all Yu and irregular-Hodge indices are rational.  An \emph{NC
divisor} on a smooth Deligne--Mumford stack is one that becomes simple
normal crossing on an \'etale scheme atlas.  A \emph{labelled SNC
divisor} is an NC divisor with a finite global presentation by smooth
effective Cartier divisors whose multiple intersections are smooth of
the expected codimension; the labels rule out permutation of local
branches.  Rounding of an integral divisor is coefficientwise:
\[
 \left\lfloor t\sum_i e_i\mathscr E_i\right\rfloor
 :=\sum_i\lfloor te_i\rfloor\mathscr E_i
 \qquad(t\in\mathbb Q).
\]

\subsection{Compactifications and Yu lattices}
\label{subsec:compactifications-yu-lattices}

\begin{definition}[Landau--Ginzburg model and good compactification]
\label{def:good-stack-compactification}
A smooth Deligne--Mumford Landau--Ginzburg model is a pair
\((\mathscr U,w)\), where \(\mathscr U\) is smooth and
\(w\colon\mathscr U\to\A^1\) is regular.  A \emph{good stack
compactification} consists of a smooth proper Deligne--Mumford stack
\(\mathscr X\), a dense open immersion
\(\mathscr U\subset\mathscr X\), a labelled SNC boundary
\[
 \mathscr D=(\mathscr X\setminus\mathscr U)_{\mathrm{red}},
\]
and a morphism \(\overline w\colon\mathscr X\to\Pj^1\) extending
\(w\).  It is \emph{projective good} if the coarse moduli space of
\(\mathscr X\) is projective.
\end{definition}

\begin{definition}[NC rational stack compactification]
\label{def:good-rational-stack-compactification}
An \emph{NC rational stack compactification} of \((\mathscr U,w)\)
consists of a smooth proper Deligne--Mumford stack \(\mathscr X\), a
dense open immersion \(\mathscr U\subset\mathscr X\), the NC boundary
\(\mathscr D=(\mathscr X\setminus\mathscr U)_{\mathrm{red}}\), and a
rational extension
\(\overline w\colon\mathscr X\dashrightarrow\Pj^1\).  On every
component on which \(w\) is not identically zero, write, with no
common prime component,
\[
 \divisor(w)=\mathscr Z-\mathscr P,
 \qquad |\mathscr P|\subseteq|\mathscr D|;
\]
on a component on which \(w=0\), put \(\mathscr Z=\mathscr P=0\).
There must be a neighbourhood \(\mathscr V\) of \(|\mathscr P|\) on
which \(\mathscr Z\) is empty or smooth and
\(\mathscr Z+\mathscr D\) is reduced NC.  Equivalently, on a suitable
\'etale chart near the polar locus, \(w\) has one of the forms
\[
 \frac{1}{x_1^{e_1}\cdots x_r^{e_r}}
 \qquad\text{or}\qquad
 \frac{z}{x_1^{e_1}\cdots x_r^{e_r}}.
\]
When the boundary is labelled SNC we say so explicitly.  The
compactification is \emph{projective} if its coarse moduli space is
projective.
\end{definition}

This is \cite[Definition~3.6]{WangCompactification}.
For a good compactification put
\(\mathscr P=\overline w^*(\infty)\); for an NC rational
compactification use the polar divisor in the preceding definition.
In either case set
\begin{align}
 K_{\mathscr X,w}^\bullet
 &:={}
 \left(
  \Omega_{\mathscr X}^\bullet(\log\mathscr D)(*\mathscr P),
  d+dw\wedge
 \right),
\label{eq:companion-rational-ambient-complex}\\
 F^\lambda K_{\mathscr X,w}^a
 &:={}
 \begin{cases}
  0,&a<\lceil\lambda\rceil,\\[3pt]
  \Omega_{\mathscr X}^a(\log\mathscr D)
  \bigl(\lfloor(a-\lambda)\mathscr P\rfloor\bigr),
  &a\geq\lceil\lambda\rceil.
 \end{cases}
\label{eq:companion-rational-yu-lattice}
\end{align}
We call \(F^\lambda K_{\mathscr X,w}^\bullet\) the \emph{Yu lattice
at level \(\lambda\)}.  It is the stack version of Yu's lattice
\cite[\S1(b), equations~(5)--(7)]{Yu}.  The local forms above give
\(dw\in\Omega_{\mathscr X}^1(\log\mathscr D)(\mathscr P)\), so every
\(F^\lambda K_{\mathscr X,w}^\bullet\) is a subcomplex.
Set
\begin{equation}
 F^{>\lambda}K_{\mathscr X,w}^\bullet
 :=\bigcup_{\mu>\lambda}F^\mu K_{\mathscr X,w}^\bullet.
\label{eq:companion-yu-lattice-above-lambda}
\end{equation}
We call \(\lambda\) a \emph{lattice jump} if
\(F^{>\lambda}K_{\mathscr X,w}^\bullet
 \ne F^\lambda K_{\mathscr X,w}^\bullet\).  Formula
\eqref{eq:companion-rational-yu-lattice} shows that the set of lattice
jumps is locally finite in \(\mathbb R\).

\subsection{Comparison and degeneration theorems}
\label{subsec:comparison-degeneration-theorems}

\begin{theorem}[Compactification comparison]
\label{thm:companion-compactification-comparison}
Let \((\mathscr U,w)\) be a smooth separated complex
Deligne--Mumford Landau--Ginzburg model.
\begin{enumerate}[label=\textup{(\roman*)}]
\item Every NC rational stack compactification has a restriction
  quasi-isomorphism
  \begin{equation}
   R\Gamma(\mathscr X,K_{\mathscr X,w}^\bullet)
   \xrightarrow{\ \sim\ }
   R\Gamma\!\left(
    \mathscr U,(\Omega_{\mathscr U}^\bullet,d+dw\wedge)
   \right).
  \label{eq:nc-rational-total-logarithmic-comparison}
  \end{equation}
\item For every \(\lambda\in\mathbb Q\), comparison of any two good
  compactifications identifies the images of the hypercohomology of
  their Yu lattices inside
  \(H_{\mathrm{dR}}^k(\mathscr U,w)\).  Comparison of an NC rational
  compactification with a good compactification resolving its rational
  map identifies the image of its Yu lattice with the same subspace.
\item For fixed \(\alpha\in\mathbb Q\cap[0,1)\), these comparisons are
  compatible with all transition maps among the levels
  \(p-\alpha\), \(p\in\mathbb Z\).
\item The common image is the whole twisted de Rham group when
  \(\lambda\leq0\).
\end{enumerate}
\end{theorem}

Here \textup{(i)} is
\cite[Proposition~3.1]{WangCompactification}, and
\textup{(ii)--(iv)} collect
\cite[Proposition~3.5, Proposition~3.10, Corollary~3.11, and
Theorem~5.6]{WangCompactification} in the form used below.

\begin{definition}[Canonical irregular Hodge filtration]
\label{def:canonical-irregular-hodge-filtration}
Denote the common image in
\cref{thm:companion-compactification-comparison} by
\[
 F_{\mathrm{irr}}^\lambda
 H_{\mathrm{dR}}^k(\mathscr U,w).
\]
Put
\[
 F_{\mathrm{irr}}^{>\lambda}
 :=\bigcup_{\mu>\lambda}F_{\mathrm{irr}}^\mu,
 \qquad
 \Gr_{F_{\mathrm{irr}}}^{\lambda}
 :=F_{\mathrm{irr}}^\lambda/F_{\mathrm{irr}}^{>\lambda}.
\]
We call \(\Gr_{F_{\mathrm{irr}}}^{\lambda}\) the \emph{rational graded
piece} at \(\lambda\).
Thus
\begin{equation}
 F_{\mathrm{irr}}^\lambda
 H_{\mathrm{dR}}^k(\mathscr U,w)
 =H_{\mathrm{dR}}^k(\mathscr U,w)
 \qquad(\lambda\leq0).
\label{eq:canonical-nonpositive-level-comparison}
\end{equation}
\end{definition}

Here a \emph{Yu lattice} is a filtered subcomplex on a chosen
compactification; the corresponding irregular Hodge level is its
compactification-independent hypercohomological image in twisted de
Rham cohomology.  The comparison theorem identifies this image, and
\(E_1\)-degeneration makes the defining map injective.

\begin{theorem}[Fixed-fractional-part \(E_1\)-degeneration]
\label{thm:companion-fixed-alpha-e1-degeneration}
Assume that the coarse moduli space of \(\mathscr U\) is
quasiprojective, and let
\((\mathscr X,\mathscr D,\overline w)\) be an NC rational stack
compactification.  For every
\(\alpha\in\mathbb Q\cap[0,1)\), the fixed-\(\alpha\) Kontsevich
spectral sequence degenerates at \(E_1\).  Equivalently, the spectral
sequence of the integer-indexed Yu filtration
\[
 F_\alpha^p:=F^{p-\alpha}
 \qquad(p\in\mathbb Z)
\]
degenerates at \(E_1\).
\end{theorem}

This is the NC rational degeneration theorem
\cite[Theorem~1.2]{WangEOne}.

\begin{remark}[Different fractional parameters]
\label{rem:fractional-parameters-levelwise-comparison}
The fixed-\(\alpha\) qualification in
\cref{thm:companion-compactification-comparison} is essential:
comparison morphisms belonging to different fractional parameters
are not asserted to commute with transition maps
\cite[Remark~3.12]{WangCompactification}.  Accordingly, two rational
levels with different fractional parts are compared only after their
hypercohomology groups have been embedded separately in the common
group \(H_{\mathrm{dR}}^k(\mathscr U,w)\).  The rational graded piece
at \(\lambda\) is therefore formed from the separately identified
images at \(\lambda\) and immediately above \(\lambda\), rather than
from a single fixed-\(\alpha\) spectral sequence.
\end{remark}

\subsection{Level and graded comparison on NC rational compactifications}
\label{subsec:cohomological-consequences-companion-inputs}

\begin{proposition}[Level and graded comparison on an NC rational
compactification]
\label{prop:nc-rational-level-graded-comparison}
Under the hypotheses of
\cref{thm:companion-fixed-alpha-e1-degeneration}, the following hold.
\begin{enumerate}[label=\textup{(\roman*)}]
\item For every \(k\) and \(\lambda\in\mathbb Q\), restriction
  induces an injection
  \begin{equation}
   \bH^k\!\left(
    \mathscr X,F^\lambda K_{\mathscr X,w}^\bullet
   \right)
   \lhook\joinrel\longrightarrow
   H_{\mathrm{dR}}^k(\mathscr U,w),
  \label{eq:nc-rational-level-injection}
  \end{equation}
  whose image is
  \(F_{\mathrm{irr}}^\lambda
    H_{\mathrm{dR}}^k(\mathscr U,w)\).
\item The inclusions at the levels \(\lambda\) and \(>\lambda\)
  induce a canonical isomorphism
  \begin{equation}
   \bH^k\!\left(
    \mathscr X,
    F^\lambda K_{\mathscr X,w}^\bullet/
    F^{>\lambda}K_{\mathscr X,w}^\bullet
   \right)
   \xrightarrow{\ \sim\ }
   \Gr_{F_{\mathrm{irr}}}^{\lambda}
   H_{\mathrm{dR}}^k(\mathscr U,w).
  \label{eq:rational-graded-comparison}
  \end{equation}
\item For every \(\lambda\leq0\), restriction is a quasi-isomorphism
  \begin{equation}
   R\Gamma\!\left(
    \mathscr X,F^\lambda K_{\mathscr X,w}^\bullet
   \right)
   \xrightarrow{\ \sim\ }
   R\Gamma\!\left(
    \mathscr U,(\Omega_{\mathscr U}^\bullet,d+dw\wedge)
   \right).
  \label{eq:nc-rational-nonpositive-level-comparison}
  \end{equation}
\end{enumerate}
\end{proposition}

\begin{proof}
Write \(\lambda=p-\alpha\), with \(p\in\mathbb Z\) and
\(\alpha\in\mathbb Q\cap[0,1)\).  The fixed-\(\alpha\) Yu filtration
is identified with the stupid filtration of the bounded Kontsevich
complex.  Hence the \(E_1\)-degeneration in
\cref{thm:companion-fixed-alpha-e1-degeneration} implies that the
cohomology of each filtered subcomplex injects into the cohomology of
the total complex.  Combining this with the total
comparison in
\cref{thm:companion-compactification-comparison}(i) gives the map in
\eqref{eq:nc-rational-level-injection}; part~(ii) of the same
theorem identifies its image with the canonical irregular Hodge
level.  This proves \textup{(i)}.

The Yu filtration has locally finite lattice jumps.  Choose
\(\lambda^+>\lambda\) sufficiently close to \(\lambda\) that
\(F^{>\lambda}K=F^{\lambda^+}K\).  By \textup{(i)}, the
maps at \(\lambda^+\) and \(\lambda\) are injections in every
cohomological degree into the same total
de Rham group, and the internal transition identifies their images
with
\(F_{\mathrm{irr}}^{>\lambda}\subseteq
F_{\mathrm{irr}}^\lambda\).  The long exact sequence of
\[
 0\longrightarrow F^{\lambda^+}K
 \longrightarrow F^\lambda K
 \longrightarrow F^\lambda K/F^{\lambda^+}K
 \longrightarrow0
\]
therefore reduces to the quotient isomorphism in \textup{(ii)}.  This
uses the separately identified images at \(\lambda^+\) and
\(\lambda\), and remains valid
when \(\lambda\) and \(\lambda^+\) have different fractional parts,
as required by
\cref{rem:fractional-parameters-levelwise-comparison}.

Finally, for \(\lambda\leq0\), part~\textup{(i)} and
\eqref{eq:canonical-nonpositive-level-comparison} show that
restriction induces an isomorphism in every cohomological degree.
This proves \textup{(iii)}.
\end{proof}
 \section{Perfect Direct Images and Base Change for
\texorpdfstring{\(S\)}{S}-Linear Complexes}
\label{sec:perfect-s-linear-complexes}

This section isolates the finiteness arguments that do not use
differential operators.  We first treat a single coherent sheaf by
passing to the coarse moduli space.  Bounded \(S\)-linear complexes are
then reduced to this one-term case by stupid truncations.

\begin{definition}[Proper flat tame family]
\label{def:proper-flat-tame-family}
A \emph{proper flat tame family} over a finite-type complex scheme
\(S\) is a morphism
\[
 \pi\colon\mathfrak X\longrightarrow S
\]
which is proper, flat, and of finite presentation, where
\(\mathfrak X\) is a separated tame Deligne--Mumford stack.
\end{definition}

Let
\[
 \mathfrak X\xrightarrow{q}X\xrightarrow{p}S,
 \qquad
 \pi=p\circ q,
\]
be the factorization through the coarse moduli space.  If
\(g\colon T\to S\) is a morphism of finite-type complex schemes, put
\(\mathfrak X_T:=\mathfrak X\times_ST\) and \(X_T:=X\times_ST\).
Tameness identifies \(X_T\) with the coarse moduli space of
\(\mathfrak X_T\), and the resulting diagram
\begin{equation}
\begin{tikzcd}[ampersand replacement=\&]
	{\mathfrak X_T} \& {X_T} \& T \\
	{\mathfrak X} \& X \& S
	\arrow["{q_T}", from=1-1, to=1-2]
	\arrow["{g_{\mathfrak X}}"', from=1-1, to=2-1]
	\arrow["{p_T}", from=1-2, to=1-3]
	\arrow["{g_X}"', from=1-2, to=2-2]
	\arrow["g", from=1-3, to=2-3]
	\arrow["q", from=2-1, to=2-2]
	\arrow["p", from=2-2, to=2-3]
\end{tikzcd}
\label{eq:coarse-space-arbitrary-base-change-diagram}
\end{equation}
has cartesian squares.
The induced morphisms \(p\) and \(p_T\) are proper and of finite
presentation.

\begin{lemma}[Perfect direct image and proper base change for a flat
coherent sheaf]
\label{lem:s-linear-perfect-direct-image-flat-coherent-sheaf}
\label{lem:part-iii-pulled-back-vector-bundle-perfect-base-change}
\label{lem:s-linear-proper-base-change-flat-coherent-sheaf}
Let \(\pi\colon\mathfrak X\to S\) be a proper flat tame family and let
\(\mathcal E\) be a coherent \(\mathcal O_{\mathfrak X}\)-module flat
over \(S\).  Then \(q_*\mathcal E\) is a coherent
\(\mathcal O_X\)-module flat over \(S\), and
\begin{equation}
 R\pi_*\mathcal E
 \in D_{\mathrm{perf}}(S).
\label{eq:part-iii-linear-usual-pushforward-one-term}
\end{equation}
For every finite-type morphism \(g\colon T\to S\), the canonical map
\begin{equation}
 Lg^*R\pi_*\mathcal E
 \xrightarrow{\ \sim\ }
 R(\pi_T)_*g_{\mathfrak X}^*\mathcal E
\label{eq:part-iii-termwise-arbitrary-proper-base-change}
\end{equation}
is an isomorphism.  These isomorphisms are compatible with composition
of base changes and with morphisms of coherent \(S\)-flat sheaves.
\end{lemma}

\begin{proof}
For a tame stack, \(q_*\) is exact on quasi-coherent sheaves and the
formation of the coarse moduli space and of \(q_*\) commutes with
arbitrary base change
\cite[Definition~3.1, Theorem~3.2, and Corollary~3.3(a)]{AOV}.
In particular, \(q_*\mathcal E\) is coherent.  It is flat over \(S\):
for an \(\mathcal O_S\)-module \(M\), the projection formula gives
\[
 q_*\mathcal E\otimes p^*M
 \simeq
 q_*\bigl(\mathcal E\otimes\pi^*M\bigr),
\]
and the functors on the right are exact in \(M\).

Since \(q_*\) is exact,
\[
 R\pi_*\mathcal E
 \simeq Rp_*Rq_*\mathcal E
 \simeq Rp_*q_*\mathcal E.
\]
The algebraic-space form of proper flat cohomology now gives both
perfectness and arbitrary derived base change
\cite[\StacksTag{0CTM}]{Stacks}; its scheme precursor is
\cite[\StacksTag{07VK}]{Stacks}.  Using the compatibility of
\(q_*\) with the two cartesian squares in
\eqref{eq:coarse-space-arbitrary-base-change-diagram} identifies the
resulting base-change isomorphism with
\eqref{eq:part-iii-termwise-arbitrary-proper-base-change}.
Functoriality follows from the same construction.
\end{proof}

\begin{definition}[\(S\)-linear complexes and indexed systems]
\label{def:s-linear-complex}
\label{def:indexed-s-linear-complexes}
Let \(\pi\colon\mathfrak X\to S\) be a morphism.  An
\emph{\(S\)-linear complex} on \(\mathfrak X/S\) is a graded
\(\mathcal O_{\mathfrak X}\)-module \(\mathcal C^\bullet\) equipped
with \(\pi^{-1}\mathcal O_S\)-linear maps
\[
 d_{\mathcal C}^a\colon
 \mathcal C^a\longrightarrow\mathcal C^{a+1},
 \qquad
 d_{\mathcal C}^{a+1}d_{\mathcal C}^a=0.
\]
Morphisms are \(\pi^{-1}\mathcal O_S\)-linear chain maps.

Let \(\Lambda\) be an ordered set.  An \emph{indexed system of bounded
\(S\)-linear complexes with coherent \(S\)-flat terms} consists of a
proper flat tame family \(\pi\colon\mathfrak X\to S\), such a complex
\(\mathcal C^{\lambda,\bullet}\) for every \(\lambda\in\Lambda\), and
transition morphisms
\[
 \theta_{\mu,\lambda}\colon
 \mathcal C^{\mu,\bullet}\longrightarrow
 \mathcal C^{\lambda,\bullet}
 \qquad(\mu\geq\lambda)
\]
which are the identity for \(\mu=\lambda\) and are compatible with
composition.  Every term \(\mathcal C^{\lambda,a}\) is required to be
a coherent \(\mathcal O_{\mathfrak X}\)-module flat over \(S\).
\end{definition}

The structural morphism factors in ringed sites as
\begin{equation}
 (\mathfrak X,\mathcal O_{\mathfrak X})
 \xrightarrow{u}
 (\mathfrak X,\pi^{-1}\mathcal O_S)
 \xrightarrow{v}
 (S,\mathcal O_S),
 \qquad
 \pi=v\circ u.
\label{eq:factorization-ringed-sites}
\end{equation}
For a complex \(\mathcal C^\bullet\) of \(\pi^{-1}\mathcal O_S\)-modules, set
\begin{equation}
 R\pi_{*,S}\mathcal C^\bullet
 :=
 Rv_*\mathcal C^\bullet
 \in D(\mathcal O_S).
\label{eq:s-linear-derived-direct-image-notation}
\end{equation}
The direct image \(u_*\) is restriction of scalars and is exact.
For an \(\mathcal O_{\mathfrak X}\)-module \(\mathcal E\), write
\[
\operatorname{Res}_S\mathcal E
 :=u_*\mathcal E
 =Ru_*\mathcal E.
\]
Flatness of \(\pi\) makes the left adjoint \(u^*\) exact.
Hence \(u_*\) preserves injectives, and
\begin{equation}
 R\pi_{*,S}\operatorname{Res}_S\mathcal E
 \simeq
 R(v\circ u)_*\mathcal E
 =
 R\pi_*\mathcal E.
\label{eq:s-linear-restricted-usual-direct-image}
\end{equation}

The factorization \eqref{eq:factorization-ringed-sites} and base change
fit into the commutative diagram
\begin{equation}
\begin{tikzcd}
 {(\mathfrak X_T,\mathcal O_{\mathfrak X_T})}
  \arrow[r,"u_T"]
  \arrow[d,"g_{\mathfrak X}"']
 &
 {(\mathfrak X_T,\pi_T^{-1}\mathcal O_T)}
  \arrow[r,"v_T"]
  \arrow[d,"g_{\mathrm{sc}}"]
 &
 {(T,\mathcal O_T)}
  \arrow[d,"g"]
 \\
 {(\mathfrak X,\mathcal O_{\mathfrak X})}
  \arrow[r,"u"']
 &
 {(\mathfrak X,\pi^{-1}\mathcal O_S)}
  \arrow[r,"v"']
 &
 {(S,\mathcal O_S).}
\end{tikzcd}
\label{eq:s-linear-ringed-site-factorization}
\end{equation}
The morphism \(g_{\mathrm{sc}}\) in
\eqref{eq:s-linear-ringed-site-factorization} defines derived scalar
extension \(Lg_{\mathrm{sc}}^*\mathcal C^\bullet\).
If \(\mathcal E\) is flat over \(S\), then
\begin{equation}
 Lg_{\mathrm{sc}}^*\operatorname{Res}_S\mathcal E
 \simeq
 \operatorname{Res}_Tg_{\mathfrak X}^*\mathcal E.
\label{eq:s-linear-scalar-extension-restriction-comparison}
\end{equation}
Consequently, scalar extension of an \(S\)-linear complex with
coherent \(S\)-flat terms is computed termwise and has coherent
\(T\)-flat terms.

\begin{theorem}[Perfect direct images and proper base change for
\(S\)-linear complexes]
\label{thm:perfect-direct-images-s-linear-complexes}
\label{thm:part-iii-perfect-direct-image-arbitrary-base-change}
Let
\(\{\mathcal C^{\lambda,\bullet}\}_{\lambda\in\Lambda}\) be an
indexed system of bounded \(S\)-linear complexes with coherent
\(S\)-flat terms.
Put
\begin{equation}
 \mathcal R_S^\lambda
 :=
 R\pi_{*,S}\mathcal C^{\lambda,\bullet}.
\label{eq:abstract-perfect-direct-image-level}
\end{equation}
Then:
\begin{enumerate}[label=\textup{(\roman*)}]
\item
Every \(\mathcal R_S^\lambda\) belongs to
\(D_{\mathrm{perf}}(S)\).
\item
For every finite-type morphism \(g\colon T\to S\), the canonical map
\begin{equation}
 Lg^*\mathcal R_S^\lambda
 \xrightarrow{\ \sim\ }
 R(\pi_T)_{*,T}
 Lg_{\mathrm{sc}}^*\mathcal C^{\lambda,\bullet}
\label{eq:part-iii-level-perfect-arbitrary-base-change}
\end{equation}
is an isomorphism.  These isomorphisms are compatible with composition
of base changes and with every transition morphism
\(\theta_{\mu,\lambda}\).
\end{enumerate}
\end{theorem}

\begin{proof}
Fix \(\lambda\) and filter
\(\mathcal C^{\lambda,\bullet}\), as a complex of \(\pi^{-1}\mathcal O_S\)-modules, by
stupid truncations.  There are short exact sequences
\begin{equation}
 0\longrightarrow
 \sigma_{\geq a+1}
 \mathcal C^{\lambda,\bullet}
 \longrightarrow
 \sigma_{\geq a}
 \mathcal C^{\lambda,\bullet}
 \longrightarrow
 \operatorname{Res}_S\mathcal C^{\lambda,a}[-a]
 \longrightarrow0.
\label{eq:part-iii-stupid-filtration-short-exact}
\end{equation}
After applying \(R\pi_{*,S}\), the resulting finite filtration has
graded pieces
\begin{equation}
 \operatorname{gr}_{\sigma}^a
 R\pi_{*,S}\mathcal C^{\lambda,\bullet}
 \simeq
 R\pi_{*,S}\operatorname{Res}_S
 \mathcal C^{\lambda,a}[-a]
 \simeq
 R\pi_*\mathcal C^{\lambda,a}[-a].
\label{eq:s-linear-direct-image-stupid-graded}
\end{equation}
They are perfect by
\cref{lem:s-linear-perfect-direct-image-flat-coherent-sheaf}.
Since perfect complexes form a triangulated subcategory, descending
through the finite system of truncation triangles proves \textup{(i)}.

Apply \(Lg^*\) to the same triangles and compare them with the
truncation triangles after scalar extension.  For each one-term
quotient, \eqref{eq:s-linear-restricted-usual-direct-image},
\eqref{eq:s-linear-scalar-extension-restriction-comparison}, and
\eqref{eq:part-iii-termwise-arbitrary-proper-base-change} identify the
base-change map with an isomorphism
\[
 Lg^*R\pi_{*,S}\operatorname{Res}_S\mathcal C^{\lambda,a}
 \xrightarrow{\ \sim\ }
 R(\pi_T)_{*,T}
 Lg_{\mathrm{sc}}^*\operatorname{Res}_S\mathcal C^{\lambda,a}.
\]
Descending induction and two-out-of-three therefore prove
\textup{(ii)}.  Every construction is functorial, which gives the
stated compatibilities.
\end{proof}

\begin{corollary}[Derived fibres and transition cones]
\label{cor:abstract-flat-complex-derived-fibres-cones}
Let \(s\to S\) be a geometric point.  There are canonical
isomorphisms
\[
 Li_s^*\mathcal R_S^\lambda
 \simeq
 R\Gamma
 \left(\mathfrak X_s,
       Li_{s,\mathrm{sc}}^*\mathcal C^{\lambda,\bullet}\right).
\]
Moreover, for \(\mu\geq\lambda\), the cone of
\(\mathcal R_S^\mu\to\mathcal R_S^\lambda\) is perfect and commutes
with arbitrary finite-type base change.
\end{corollary}

\begin{proof}
The fibre formula is
\cref{thm:perfect-direct-images-s-linear-complexes} with \(T=s\).
The same theorem makes the cone perfect, and its base-change assertion
follows by applying derived pullback to the transition triangle and using
\eqref{eq:part-iii-level-perfect-arbitrary-base-change} on its first
two vertices.
\end{proof}

The base-change formula reduces geometric questions about the direct
images to their derived fibres.  We therefore record two criteria for
perfect complexes.
We use the same notation \(Y\) for the base, \(y\in Y\) for a point,
and \(\kappa(y)\) for its residue field throughout.

\begin{lemma}[Fibrewise conservativity for perfect
complexes]
\label{lem:part-iii-perfect-fibrewise-isomorphism-criterion}
Let \(Y\) be a scheme and let
\(\varphi\colon K\to L\) be a morphism in
\(D_{\mathrm{perf}}(Y)\).  Then \(\varphi\) is an isomorphism if and
only if its derived fibre
\[
 \varphi\otimes_{\mathcal O_Y}^L\kappa(y)
\]
is an isomorphism for every point \(y\in Y\).  Equivalently, it is
enough to test the derived pullback to every geometric point of
\(Y\).
\end{lemma}

\begin{proof}
The forward implication follows from functoriality of derived
pullback.  Conversely, let
\(Q:=\operatorname{Cone}(\varphi)\).  Then \(Q\) is perfect, and the
hypothesis says that
\(Q\otimes_{\mathcal O_Y}^L\kappa(y)=0\) for every \(y\in Y\).
Testing geometric points is equivalent, since extension from a
residue field to an algebraic closure is faithfully flat.

Apply \cite[\StacksTag{0BDK}]{Stacks} to \(Q\), with
\(i=0\) and rank \(r=0\).  It gives an open subscheme \(V\subset Y\)
characterized by the property that a morphism \(f\colon Z\to Y\)
factors through \(V\) if and only if \(Lf^*Q\) is the zero complex.
Every point of \(Y\) therefore factors through \(V\), so \(V=Y\).
The same cited lemma now gives \(Q=0\), so \(\varphi\) is
an isomorphism.
\end{proof}

\begin{lemma}[Fibre cohomology of a perfect complex]
\label{lem:part-iii-constant-perfect-fibre-cohomology}
Let \(Y\) be a scheme and let \(K\in D_{\mathrm{perf}}(Y)\).  For
every \(b\), the function
\[
 h_K^b\colon
 |Y|\longrightarrow\mathbb Z_{\geq0},
 \qquad
 y\longmapsto
 \dim_{\kappa(y)}
 H^b\!\left(K\otimes_{\mathcal O_Y}^L\kappa(y)\right)
\]
is upper semicontinuous.

If, in addition, \(Y\) is reduced and every \(h_K^b\) is locally
constant, then every \(\mathcal H^b(K)\) is finite locally free, the
canonical maps
\begin{equation}
 f^*\mathcal H^b(K)
 \xrightarrow{\ \sim\ }
 \mathcal H^b(Lf^*K)
\label{eq:part-iii-constant-perfect-cohomology-base-change}
\end{equation}
are isomorphisms for arbitrary morphisms \(f\colon Z\to Y\), and
locally on \(Y\) there is an isomorphism
\[
 K
 \simeq
 \bigoplus_b\mathcal H^b(K)[-b].
\]
\end{lemma}

\begin{proof}
Upper semicontinuity is
\cite[\StacksTag{0BDI}]{Stacks}.  Suppose now that \(Y\) is reduced
and every \(h_K^b\) is locally constant.
Fix \(y\in Y\), and let the common local fibre dimensions be
\(d_b\).  The local normal-form result
\cite[\StacksTag{0BCD}]{Stacks} says precisely
that a perfect complex over a ring, at a chosen prime, is
represented after localization by a finite complex whose term in
degree \(b\) is free of rank
\[
 \dim_{\kappa(y)}
 H^b(K\otimes^L\kappa(y)).
\]
Thus, after shrinking around \(y\), represent \(K\) by
\[
 \cdots
 \longrightarrow
 \mathcal O_Y^{d_{b-1}}
 \xrightarrow{d^{b-1}}
 \mathcal O_Y^{d_b}
 \xrightarrow{d^b}
 \mathcal O_Y^{d_{b+1}}
 \longrightarrow\cdots.
\]

At every point \(z\) of this neighbourhood, the middle vector
space has dimension \(d_b\), while its cohomology also has
dimension \(d_b\).  The equality
\[
 \dim H^b
 =
 d_b-\operatorname{rank}(d^{b-1}_z)
    -\operatorname{rank}(d^b_z)
\]
forces both displayed ranks to be zero.  Hence every matrix entry
of every differential vanishes at every point.  Since \(Y\) is
reduced, all these entries are zero.  The local complex therefore
has zero differential.  Its terms are its cohomology sheaves,
which proves local freeness, the local direct-sum decomposition,
and
\eqref{eq:part-iii-constant-perfect-cohomology-base-change}.
\end{proof}

For later use, we record the consequence of having only finitely many
nontrivial transitions.  The system then stabilizes toward
\(-\infty\), so its homotopy colimit is represented by any sufficiently
low level, and the preceding perfectness, base-change, and fibre
formulas extend to the limiting object and the transition cones.

\begin{lemma}[Finite-support systems of \(S\)-linear complexes]
\label{lem:finite-support-s-linear-system}
Let
\(\{\mathcal C^{\lambda,\bullet}\}_{\lambda\in\Lambda}\) be an
indexed system of bounded \(S\)-linear complexes with coherent
\(S\)-flat terms.  Choose a strictly increasing sequence
\[
 \cdots<\lambda_{i-1}<\lambda_i<\lambda_{i+1}<\cdots
 \qquad(i\in\mathbb Z)
\]
which is cofinal toward \(-\infty\), and put
\begin{equation}
 \mathcal R_i
 :=R\pi_{*,S}\mathcal C^{\lambda_i,\bullet},
 \qquad
 \mathcal G_i
 :=\operatorname{Cone}(\mathcal R_{i+1}\longrightarrow\mathcal R_i),
 \qquad
 \mathcal R_{-\infty}
 :=\operatorname*{hocolim}_{i\to-\infty}\mathcal R_i.
\label{eq:katz-oda-perfect-system}
\end{equation}
Write
\(\jmath_i\colon\mathcal R_i\to\mathcal R_{-\infty}\) for the
cocone morphisms.  Assume that \(\mathcal G_i\) is acyclic for all
but finitely many \(i\).  Then:
\begin{enumerate}[label=\textup{(\roman*)}]
\item
Every \(\mathcal R_i\), \(\mathcal G_i\), and
\(\mathcal R_{-\infty}\) is perfect.  For all sufficiently negative
\(i\), the cocone morphism
\begin{equation}
 \jmath_i\colon
 \mathcal R_i\xrightarrow{\ \sim\ }\mathcal R_{-\infty}
\label{eq:katz-oda-low-representative}
\end{equation}
is an isomorphism.

\item
These objects commute with arbitrary finite-type derived base
change.  More precisely, for \(g\colon T\to S\), put
\[
 \mathcal R_{T,i}
 :=R(\pi_T)_{*,T}
   Lg_{\mathrm{sc}}^*\mathcal C^{\lambda_i,\bullet},
 \qquad
 \mathcal G_{T,i}
 :=\operatorname{Cone}(\mathcal R_{T,i+1}\to\mathcal R_{T,i}).
\]
Then there are canonical isomorphisms
\[
 Lg^*\mathcal R_i\simeq\mathcal R_{T,i},
 \qquad
 Lg^*\mathcal G_i\simeq\mathcal G_{T,i},
 \qquad
 Lg^*\mathcal R_{-\infty}
 \simeq
 \operatorname*{hocolim}_{i\to-\infty}\mathcal R_{T,i}.
\]

\item
For every geometric point \(s\to S\), there are canonical derived
fibre identifications
\begin{align}
 Li_s^*\mathcal R_i
 &\simeq
 R\Gamma\!\left(
  \mathfrak X_s,
  Li_{s,\mathrm{sc}}^*\mathcal C^{\lambda_i,\bullet}
 \right),
\label{eq:katz-oda-derived-level-fibre}\\
 Li_s^*\mathcal G_i
 &\simeq
 \operatorname{Cone}\!\left\{
  R\Gamma\!\left(
   \mathfrak X_s,
   Li_{s,\mathrm{sc}}^*\mathcal C^{\lambda_{i+1},\bullet}
  \right)
  \longrightarrow
  R\Gamma\!\left(
   \mathfrak X_s,
   Li_{s,\mathrm{sc}}^*\mathcal C^{\lambda_i,\bullet}
 \right)
 \right\},
\label{eq:katz-oda-derived-cone-fibre}
\end{align}
\end{enumerate}
\end{lemma}

\begin{proof}
\Cref{thm:perfect-direct-images-s-linear-complexes,cor:abstract-flat-complex-derived-fibres-cones}
give perfectness, derived base change, and the first two fibre
identifications for every level and transition cone.  Choose
\(i_-\leq i_+\) such that \(\mathcal G_i\simeq0\) for
\(i<i_-\) and \(i>i_+\).  The transition maps below \(i_-\) are
isomorphisms, so cofinality gives
\(\mathcal R_i\simeq\mathcal R_{-\infty}\) for \(i\leq i_-\).
This proves the remaining perfectness assertion.

Derived pullback preserves homotopy colimits.  Applying it to
\eqref{eq:katz-oda-perfect-system} and using levelwise base change
gives the asserted formula for
\(Lg^*\mathcal R_{-\infty}\).  The cone formula and all
compatibilities follow from functoriality.
\end{proof}
 \section{De Rham DG Modules and the Katz--Oda Filtration}
\label{sec:de-rham-dg-modules-katz-oda}

The geometric complexes used below are naturally modules over de
Rham differential graded algebras.  This language records both their
differentials and the exterior action of forms, and makes the passage
from an absolute complex to its relative quotient formal.  The
perfectness arguments have already been isolated in
\cref{sec:perfect-s-linear-complexes}; the additional de
Rham structure here produces the Katz--Oda operators and the
connection used in the cohomological exactness theorem below.

For the particular one-parameter toric Landau--Ginzburg family used
in their quasi-stable degeneration, Harder--Lee carry out the
corresponding logarithmic Gauss--Manin construction in
\cite[\S8]{HarderLee}.  The purpose of the formalism developed here is
to isolate this mechanism for general algebraic bases and general
Katz--Oda-compatible filtered de Rham dg modules, and to combine it
with perfectness, arbitrary base change, and fibrewise
\(E_1\)-degeneration.

\subsection{Relative de Rham dg modules and base change}
\label{subsec:relative-de-rham-dg-modules}
\label{subsec:relative-de-rham-dg-module-scalar-extension}

Let \(\pi\colon\mathfrak X\to S\) be a morphism from a complex
Deligne--Mumford stack to a complex scheme.  Write
\((\Omega_{\mathfrak X/S}^\bullet,d_{\mathfrak X/S})\) for its
relative de Rham dg algebra.

\begin{definition}[Relative de Rham dg module]
\label{def:relative-de-rham-dg-module}
A \emph{relative de Rham dg module} on \(\mathfrak X/S\) is a
differential graded module \((\mathcal C^\bullet,D)\) over
\((\Omega_{\mathfrak X/S}^\bullet,d_{\mathfrak X/S})\) whose
degree-zero action is the given \(\mathcal O_{\mathfrak X}\)-module
structure.  Thus, for a local relative form \(\eta\) of degree \(r\),
\begin{equation}
 D(\eta\cdot u)
 =d_{\mathfrak X/S}\eta\cdot u+(-1)^r\eta\cdot Du.
\label{eq:relative-de-rham-dg-module-Leibniz}
\end{equation}
Morphisms are degree-zero dg-module morphisms.
\end{definition}

For a local function \(a\in\mathcal O_{\mathfrak X}\), the dg
Leibniz identity specializes to
\begin{equation}
 D(au)-aDu=d_{\mathfrak X/S}a\cdot u.
\label{eq:relative-de-rham-function-Leibniz}
\end{equation}
In particular, \(D\) is \(\pi^{-1}\mathcal O_S\)-linear.
Since the de Rham algebra is generated by functions and their
differentials, an \(\mathcal O_{\mathfrak X}\)-linear chain
map between relative de Rham dg modules automatically commutes with
the entire \(\Omega_{\mathfrak X/S}^\bullet\)-action.
Replacing \(S\) by \(\Spec\mathbb C\) gives the corresponding absolute notion.

Let \(\pi\colon\mathfrak X\to S\) be a proper flat tame family as in
\cref{def:proper-flat-tame-family}.
If \(\mathcal C^\bullet\) is a bounded relative de Rham dg module
whose terms are coherent \(\mathcal O_{\mathfrak X}\)-modules flat
over \(S\), then \eqref{eq:relative-de-rham-function-Leibniz} makes
its underlying complex a bounded \(S\)-linear complex in the sense
of \cref{def:s-linear-complex}, with coherent \(S\)-flat
terms.

Let \(g\colon T\to S\) be a finite-type morphism.
Formation of relative differentials and their
exterior powers commutes with base change, giving degreewise
identifications compatible with wedge products
\begin{equation}
 g_{\mathfrak X}^*\Omega_{\mathfrak X/S}^r
 \simeq
 \Omega_{\mathfrak X_T/T}^r
 \qquad(r\geq0).
\label{eq:relative-de-rham-dg-algebra-base-change}
\end{equation}
Thus \(Lg_{\mathrm{sc}}^*\mathcal C^\bullet\) is a
bounded relative de Rham dg module on \(\mathfrak X_T/T\) whose terms
are coherent and flat over \(T\).

\subsection{The Katz--Oda filtration and relative quotient}
\label{subsec:katz-oda-filtration-relative-quotient}
\label{subsec:absolute-relative-dg-quotient}

Let \(\pi\colon\mathfrak X\to S\) be a proper flat tame family.
From this subsection onward, assume that \(S\) is smooth over \(\mathbb C\)
and that the ordered index set
\(\Lambda\subset\mathbb Q\) is stable under translation by
nonpositive integers.

The absolute and relative de Rham dg algebras are related by
canonical dg-algebra morphisms
\begin{equation}
 \pi^{-1}(\Omega_{S/\mathbb C}^\bullet,d)
 \longrightarrow
 (\Omega_{\mathfrak X/\mathbb C}^\bullet,d)
 \longrightarrow
 (\Omega_{\mathfrak X/S}^\bullet,d_{\mathfrak X/S}).
\label{eq:absolute-relative-de-rham-dg-algebra-morphisms}
\end{equation}
Let
\(J_{\mathfrak X/S}\subset\Omega_{\mathfrak X/\mathbb C}^\bullet\)
be the dg ideal generated by the image of
\(\pi^{-1}\Omega_{S/\mathbb C}^{\geq1}\), equivalently by the image
of \(\pi^{-1}\Omega_{S/\mathbb C}^1\).  The cotangent sequence and
the universal property of relative differentials give a canonical
identification
\begin{equation}
 \Omega_{\mathfrak X/\mathbb C}^\bullet/J_{\mathfrak X/S}
 \xrightarrow{\ \sim\ }
 \Omega_{\mathfrak X/S}^\bullet
\label{eq:absolute-relative-de-rham-dg-algebra-quotient}
\end{equation}
of dg algebras.  These statements are understood on the
lisse--\'{e}tale site and may be checked on smooth charts.

\begin{lemma}[The Katz--Oda filtration and relative quotient]
\label{lem:absolute-to-relative-dg-quotient}
Let \((\mathcal A^\bullet,\mathbf D)\) be a bounded dg module over
\((\Omega_{\mathfrak X/\mathbb C}^\bullet,d)\), whose degree-zero
action is the given \(\mathcal O_{\mathfrak X}\)-module structure.
For \(r\geq0\), define
\begin{equation}
 L^r\mathcal A^\bullet
 :=\operatorname{Im}\!\left\{
  \pi^{-1}\Omega_{S/\mathbb C}^r
  \otimes_{\pi^{-1}\mathcal O_S}
  \mathcal A^\bullet[-r]
  \longrightarrow
  \mathcal A^\bullet
 \right\},
\label{eq:abstract-induced-base-degree-filtration}
\end{equation}
We call \(L\) the \emph{Katz--Oda filtration}; it records the number
of differential forms pulled back from the base.  Put
\begin{equation}
 \mathcal C^\bullet
 :=\operatorname{gr}_L^0\mathcal A^\bullet
 =\mathcal A^\bullet/L^1\mathcal A^\bullet.
\label{eq:abstract-unfiltered-relative-quotient}
\end{equation}
Suppose, for assertions \textup{(ii)}--\textup{(v)} below, that
\(\mathcal A^\bullet\) moreover carries an exhaustive decreasing
\(\Lambda\)-filtration \(F\) by subcomplexes such that the dg-module
action induces maps
\begin{equation}
 \Omega_{\mathfrak X/\mathbb C}^r
 \otimes_{\mathcal O_{\mathfrak X}}
 F^{\lambda-r}\mathcal A^\bullet[-r]
 \longrightarrow F^\lambda\mathcal A^\bullet
 \qquad(r\geq0).
\label{eq:abstract-filtered-absolute-de-rham-action}
\end{equation}
In that case, put
\[
 \mathcal A^{\lambda,\bullet}:=F^\lambda\mathcal A^\bullet,
 \qquad
 L^r\mathcal A^{\lambda,\bullet}
 :=F^\lambda\mathcal A^\bullet\cap L^r\mathcal A^\bullet,
\]
and give \(\mathcal C^\bullet\) the quotient filtration, so that
\begin{equation}
 (\mathcal C^\bullet,F)
 :=\operatorname{gr}_L^0(\mathcal A^\bullet,F),
 \qquad
 \mathcal C^{\lambda,\bullet}
 :=F^\lambda\mathcal C^\bullet
 =\operatorname{gr}_L^0\mathcal A^{\lambda,\bullet}.
\label{eq:abstract-induced-relative-system}
\end{equation}

Then the following assertions hold.
\begin{enumerate}[label=\textup{(\roman*)}]
\item
The filtration \(L\) is finite and decreasing, with
\(L^0\mathcal A^\bullet=\mathcal A^\bullet\), and every \(L^r\) is a
dg submodule.  The quotient \(\mathcal C^\bullet\) is canonically a
dg module over
\((\Omega_{\mathfrak X/S}^\bullet,d_{\mathfrak X/S})\).

\item
Every \(\mathcal A^{\lambda,\bullet}\) is an absolute de Rham dg
submodule and every \(\mathcal C^{\lambda,\bullet}\) is a relative de
Rham dg submodule.  The relative action induces maps
\begin{equation}
 \Omega_{\mathfrak X/S}^r
 \otimes_{\mathcal O_{\mathfrak X}}
 \mathcal C^{\lambda-r,\bullet}[-r]
 \longrightarrow
 \mathcal C^{\lambda,\bullet}.
\label{eq:abstract-filtered-relative-de-rham-action}
\end{equation}

\item
The filtration inclusions are absolute dg-module morphisms and induce
relative dg-module morphisms on the quotient:
\begin{equation}
 \tau_{\mu,\lambda}\colon
 \mathcal A^{\mu,\bullet}\lhook\joinrel\longrightarrow
 \mathcal A^{\lambda,\bullet},
 \qquad
 \theta_{\mu,\lambda}
 :=\operatorname{gr}_L^0(\tau_{\mu,\lambda})
 \qquad(\mu\geq\lambda).
\label{eq:abstract-filtered-level-transitions}
\end{equation}
Every \(\operatorname{gr}_L^r\mathcal A^{\lambda,\bullet}\) is
canonically a relative de Rham dg module.  Multiplication by
pulled-back base forms induces canonical relative dg-module
morphisms
\begin{equation}
 \psi_{\lambda,r}\colon
 \pi^{-1}\Omega_{S/\mathbb C}^r
 \otimes_{\pi^{-1}\mathcal O_S}
 \mathcal C^{\lambda-r,\bullet}[-r]
 \longrightarrow
 \operatorname{gr}_L^r\mathcal A^{\lambda,\bullet}
 \qquad(r\geq0).
\label{eq:abstract-canonical-base-action-to-associated-graded}
\end{equation}
For \(r=0\), this is the tautological identification.  For
\(\mu\geq\lambda\), these morphisms satisfy
\begin{equation}
 \operatorname{gr}_L^r(\tau_{\mu,\lambda})\circ\psi_{\mu,r}
 =\psi_{\lambda,r}\circ
  (\operatorname{id}\otimes\theta_{\mu-r,\lambda-r}[-r]).
\label{eq:abstract-canonical-base-action-transition-compatibility}
\end{equation}

\item
Let \(f\in\mathcal O_S\) be a local function, let
\(\eta\in\Omega_{S/\mathbb C}^r\), and let
\(v\in\mathcal C^{\lambda-r,\bullet}\).  If
\(u\in L^r\mathcal A^{\lambda,\bullet}\) represents
\(\psi_{\lambda,r}(\eta\otimes v)\), then
\([\mathbf D,\pi^*f]u\) belongs to
\(L^{r+1}\mathcal A^{\lambda,\bullet}\), and its class in the next
graded piece is
\begin{equation}
 \bigl[[\mathbf D,\pi^*f]u\bigr]
 =\psi_{\lambda,r+1}\!\left(
   df\wedge\eta\otimes
   \theta_{\lambda-r,\lambda-r-1}(v)
  \right).
\label{eq:katz-oda-base-function-action}
\end{equation}

\item
The first differential of the \(L\)-spectral sequence of
every \(\mathcal A^{\lambda,\bullet}\) obeys the de Rham Leibniz
rule.  More explicitly, for a local base \(r\)-form \(\eta\) and a
local section \([u]\) of its \(E_1\)-sheaf,
\begin{equation}
 d_1(\eta\cdot[u])
 =d\eta\cdot[u]+(-1)^r\eta\cdot d_1[u].
\label{eq:abstract-L-spectral-sequence-Leibniz}
\end{equation}
\end{enumerate}
\end{lemma}

\begin{proof}
For \textup{(i)}, the base de Rham algebra is generated in degree
one, so associativity of its restricted base action gives
\(L^{r+1}\mathcal A^\bullet\subseteq L^r\mathcal A^\bullet\).
For a local base \(r\)-form \(\eta\) and a local section \(u\), the
dg Leibniz identity
\begin{equation}
 \mathbf D(\eta\cdot u)
 =d\eta\cdot u+(-1)^r\eta\cdot\mathbf D u
\label{eq:abstract-filtered-dg-module-Leibniz}
\end{equation}
shows that \(L^r\mathcal A^\bullet\) is a subcomplex; graded
commutativity shows that it is a dg submodule.  The filtration is
finite because \(\Omega_{S/\mathbb C}^r=0\) for \(r>\dim S\).
Moreover,
\[
 L^1\mathcal A^\bullet
 =J_{\mathfrak X/S}\mathcal A^\bullet.
\]
Consequently \eqref{eq:absolute-relative-de-rham-dg-algebra-quotient}
makes \(\mathcal C^\bullet\) a dg module over the relative de Rham
algebra.

For \textup{(ii)}, the image of \(F^\lambda\mathcal A^\bullet\) in
\(\mathcal C^\bullet\) is
\[
 \frac{F^\lambda\mathcal A^\bullet}
 {F^\lambda\mathcal A^\bullet\cap L^1\mathcal A^\bullet}
 =\operatorname{gr}_L^0\mathcal A^{\lambda,\bullet},
\]
which proves \eqref{eq:abstract-induced-relative-system}.  To obtain
\eqref{eq:abstract-filtered-relative-de-rham-action}, lift a relative
form locally to an absolute form.  The filtration shift follows from
\eqref{eq:abstract-filtered-absolute-de-rham-action}; different lifts
differ by \(J_{\mathfrak X/S}\) and therefore act identically on the
quotient.  Since \(F^\lambda\subseteq F^{\lambda-r}\), the same
action condition shows that each \(\mathcal A^{\lambda,\bullet}\) is
an absolute dg submodule.

For \textup{(iii)}, the filtration inclusions preserve \(L\), and
their quotients are relative dg-module morphisms.  The restricted
base action maps
\[
 \pi^{-1}\Omega_{S/\mathbb C}^r
 \otimes F^{\lambda-r}\mathcal A^\bullet[-r]
 \longrightarrow L^r\mathcal A^{\lambda,\bullet}.
\]
Changing a lift from \(\mathcal C^{\lambda-r,\bullet}\) by an element
of \(L^1\mathcal A^{\lambda-r,\bullet}\) changes its image by an
element of \(L^{r+1}\mathcal A^{\lambda,\bullet}\).  The action thus
descends to \(\psi_{\lambda,r}\).  Positive-degree base forms raise
the \(L\)-degree, so they act trivially on every associated graded
piece; hence both the target and \(\psi_{\lambda,r}\) are relative.
Modulo \(L^{r+1}\), the first term on the right of
\eqref{eq:abstract-filtered-dg-module-Leibniz} vanishes, while the
second gives the differential on the shifted tensor product.  Thus
\(\psi_{\lambda,r}\) is a dg-module morphism.  Its compatibility with
the level maps follows from associativity of the action and
functoriality of the filtration inclusions.

For \textup{(iv)}, choose a lift \(\widetilde v\in
\mathcal A^{\lambda-r,\bullet}\) of \(v\).  Modulo
\(L^{r+1}\mathcal A^{\lambda,\bullet}\), the section \(u\) is
represented by \(\eta\wedge\widetilde v\), and the dg-module Leibniz
formula gives
\[
 [\mathbf D,\pi^*f](\eta\wedge\widetilde v)
 =df\wedge\eta\wedge\widetilde v.
\]
In base degree \(r+1\), the last factor is viewed through the
filtration inclusion
\(\mathcal A^{\lambda-r,\bullet}\hookrightarrow
\mathcal A^{\lambda-r-1,\bullet}\), which induces
\(\theta_{\lambda-r,\lambda-r-1}\).  This proves
\eqref{eq:katz-oda-base-function-action}.

Finally, apply \eqref{eq:abstract-filtered-dg-module-Leibniz} one
order higher in the \(L\)-filtration.  This gives
\eqref{eq:abstract-L-spectral-sequence-Leibniz} and proves
\textup{(v)}.
\end{proof}

The filtration \(L\) is the canonical filtration used by Katz and
Oda to construct the Gauss--Manin connection
\cite[\S2]{KatzOda}; it is also called the Koszul filtration.

\subsection{Katz--Oda-compatible filtered de Rham dg modules}
\label{subsec:katz-oda-compatible-filtered-de-rham-modules}

\begin{definition}[Katz--Oda compatibility and relative
\texorpdfstring{\(S\)}{S}-flatness]
\label{def:katz-oda-compatible-filtered-de-rham-dg-module}
A \emph{filtered absolute de Rham dg module} on
\(\mathfrak X/S\) is a filtered dg module
\((\mathcal A^\bullet,\mathbf D,F)\) satisfying the hypotheses of
\cref{lem:absolute-to-relative-dg-quotient}.  It is called
\emph{Katz--Oda-compatible} if every canonical morphism
\(\psi_{\lambda,r}\) in
\eqref{eq:abstract-canonical-base-action-to-associated-graded} is an
isomorphism.  In that case, write its inverse as
\begin{equation}
 \phi_{\lambda,r}:=\psi_{\lambda,r}^{-1}\colon
 \operatorname{gr}_L^r\mathcal A^{\lambda,\bullet}
 \xrightarrow{\ \sim\ }
 \pi^{-1}\Omega_{S/\mathbb C}^r
 \otimes_{\pi^{-1}\mathcal O_S}
 \mathcal C^{\lambda-r,\bullet}[-r]
 \qquad(r\geq0).
\label{eq:katz-oda-associated-graded}
\end{equation}
For \(r=0\), this is the tautological identification in
\eqref{eq:abstract-induced-relative-system}.

The module is called \emph{relatively \(S\)-flat} if, for every
\(\lambda\in\Lambda\) and every \(k\), the induced relative term
\(\mathcal C^{\lambda,k}\) is a coherent
\(\mathcal O_{\mathfrak X}\)-module flat over \(S\).
\end{definition}

By boundedness of \(\mathcal A^\bullet\) and
\cref{lem:absolute-to-relative-dg-quotient}, relative \(S\)-flatness
implies that, after forgetting the action of positive-degree relative
forms, the induced data
\(\{\mathcal C^{\lambda,\bullet},\theta_{\mu,\lambda}\}\) form an
indexed system of bounded \(S\)-linear complexes with coherent
\(S\)-flat terms in the sense of
\cref{def:indexed-s-linear-complexes}.

\begin{remark}[Filtered \(D\)-modules as a source of examples]
\label{rem:abstract-filtered-d-module-origin}
If \(\Lambda=\mathbb Z\) and \((\mathcal M,G_\bullet)\) is a good
filtered left
\(\mathscr D_{\mathfrak X}\)-module, then its filtered de Rham complex
\[
 \mathcal A^\bullet
 :=\DR_{\mathfrak X/\mathbb C}^{\circ}(\mathcal M),
 \qquad
 (F^\lambda\mathcal A)^m
 :=\Omega_{\mathfrak X/\mathbb C}^m
   \otimes G_{m-\lambda}\mathcal M
\]
is Katz--Oda-compatible.  It is relatively \(S\)-flat whenever its
induced relative levels have the required coherence and flatness.
Thus Katz--Oda compatibility is the standard filtered de Rham
formalism, whereas relative \(S\)-flatness is the additional family
condition used below.  This is a source of the definition, not an
input to the stacky application; see
\cite[\S\S1.1--1.3]{SaitoLogDirectImage} and \cite[\S2]{KatzOda}.
\end{remark}

\subsection{Katz--Oda operators and the limiting connection}
\label{subsec:katz-oda-limiting-connection}

\begin{proposition}[Shifted Katz--Oda operators and the limiting
connection]
\label{prop:abstract-shifted-katz-oda}
For a relatively \(S\)-flat, Katz--Oda-compatible filtered absolute
de Rham dg module, put
\begin{equation}
 K_\lambda:=R\pi_{*,S}\mathcal C^{\lambda,\bullet},
 \qquad
 K_{-\infty}:=
 \operatorname*{hocolim}_{\lambda\to-\infty}K_\lambda.
\label{eq:katz-oda-level-and-limiting-object}
\end{equation}
Write \(\jmath_\lambda\colon K_\lambda\to K_{-\infty}\) for the
cocone morphisms.  For every \(q\), the following hold.
\begin{enumerate}[label=\textup{(\roman*)}]
\item
For every \(\lambda\), there is a canonical additive operator
\begin{equation}
 \nabla_\lambda^q\colon
 \mathcal H^q(K_\lambda)
 \longrightarrow
 \Omega_{S/\mathbb C}^1\otimes
 \mathcal H^q(K_{\lambda-1}).
\label{eq:abstract-shifted-katz-oda-operator}
\end{equation}
The homotopy colimit carries a canonical algebraic integrable
connection
\begin{equation}
 \nabla_{-\infty}^q\colon
 \mathcal H^q(K_{-\infty})
 \longrightarrow
 \Omega_{S/\mathbb C}^1\otimes
 \mathcal H^q(K_{-\infty}).
\label{eq:katz-oda-limiting-connection}
\end{equation}
For \(f\in\mathcal O_S\), a local section
\(v_\lambda\in\mathcal H^q(K_\lambda)\), and a local section
\(v_{-\infty}\in\mathcal H^q(K_{-\infty})\), they satisfy
\begin{align}
 \nabla_\lambda^q(fv_\lambda)
 &=f\nabla_\lambda^q(v_\lambda)
   +df\otimes
    \mathcal H^q(\theta_{\lambda,\lambda-1})(v_\lambda),
 \label{eq:abstract-shifted-katz-oda-Leibniz}\\
 \nabla_{-\infty}^q(fv_{-\infty})
 &=f\nabla_{-\infty}^q(v_{-\infty})+df\otimes v_{-\infty}.
 \label{eq:katz-oda-limiting-Leibniz}
\end{align}
More precisely, the first differential of the \(L\)-spectral sequence
at level \(\lambda\) is
\begin{equation}
 d_{1,\lambda}(\eta\otimes v)
 =d\eta\otimes
  \mathcal H^q(\theta_{\lambda-r,\lambda-r-1})(v)
  +(-1)^r\eta\wedge\nabla_{\lambda-r}^q(v),
\label{eq:abstract-shifted-L-spectral-sequence-differential}
\end{equation}
for \(\eta\in\Omega_{S/\mathbb C}^r\) and
\(v\in\mathcal H^q(K_{\lambda-r})\).  Taking the homotopy colimit over
the directed system of levels identifies these first differentials with
the de Rham differential determined by \(\nabla_{-\infty}^q\), whose
square is zero.

\item
If \(\mu\geq\lambda\), then
the following squares commute:
\begin{equation}
\begin{tikzcd}
	{\mathcal H^q(K_\mu)} & {\Omega_{S/\mathbb C}^1\otimes\mathcal H^q(K_{\mu-1})} \\
	{\mathcal H^q(K_\lambda)} & {\Omega_{S/\mathbb C}^1\otimes\mathcal H^q(K_{\lambda-1})} \\
	{\mathcal H^q(K_{-\infty})} & {\Omega_{S/\mathbb C}^1\otimes\mathcal H^q(K_{-\infty}).}
	\arrow["{\nabla_\mu^q}", from=1-1, to=1-2]
		\arrow["{\mathcal H^q(\theta_{\mu,\lambda})}"', from=1-1, to=2-1]
	\arrow["{\operatorname{id}\otimes                 \mathcal H^q(\theta_{\mu-1,\lambda-1})}", from=1-2, to=2-2]
	\arrow["{\nabla_\lambda^q}", from=2-1, to=2-2]
	\arrow["{\mathcal H^q(\jmath_\lambda)}"', from=2-1, to=3-1]
	\arrow["{\operatorname{id}\otimes
	                 \mathcal H^q(\jmath_{\lambda-1})}", from=2-2, to=3-2]
	\arrow["{\nabla_{-\infty}^q}"', from=3-1, to=3-2]
\end{tikzcd}
\label{eq:abstract-shifted-katz-oda-compatibility}
\end{equation}

\item
If there is \(\lambda_0\in\Lambda\) such that
\(K_\lambda\to K_{\lambda-1}\) is an isomorphism for every
\(\lambda\leq\lambda_0\), then the cocone maps are isomorphisms for
all sufficiently low \(\lambda\), and \(K_{-\infty}\) is perfect.
\end{enumerate}
\end{proposition}

\begin{proof}
The first two graded pieces give a canonical short exact sequence of
complexes of sheaves of \(\mathbb C\)-vector spaces
\begin{equation}
 \begin{split}
 0\longrightarrow{}&
 \pi^{-1}\Omega_{S/\mathbb C}^1
 \otimes_{\pi^{-1}\mathcal O_S}
 \mathcal C^{\lambda-1,\bullet}[-1]
 \longrightarrow
 \mathcal A^{\lambda,\bullet}/L^2\mathcal A^{\lambda,\bullet}
 \\
 &\longrightarrow
 \mathcal C^{\lambda,\bullet}
 \longrightarrow0.
 \end{split}
\label{eq:katz-oda-first-base-degree-extension}
\end{equation}
Apply derived direct image to this sequence.  The projection formula
and local freeness of \(\Omega_{S/\mathbb C}^1\) identify the shifted
left term with
\(\Omega_{S/\mathbb C}^1\otimes K_{\lambda-1}[-1]\).  The connecting
morphism therefore has cohomological degree zero and gives
\eqref{eq:abstract-shifted-katz-oda-operator}.

For a local lift of a cohomology class, the boundary of its product
with \(f\) differs from \(f\) times its boundary by the class of
\([\mathbf D,\pi^*f]\).  After applying \(\phi_{\lambda,1}\), the
 base-action formula \eqref{eq:katz-oda-base-function-action}
identifies this class with
\(df\otimes\mathcal H^q(\theta_{\lambda,\lambda-1})(v)\), proving
\eqref{eq:abstract-shifted-katz-oda-Leibniz}.  The filtered level
inclusions and
\eqref{eq:abstract-canonical-base-action-transition-compatibility}
give a morphism between the displayed short exact sequences for
\(\mu\) and \(\lambda\).  Functoriality of connecting morphisms proves
the upper square of
\eqref{eq:abstract-shifted-katz-oda-compatibility}.

Filtered colimits are exact in the category of
\(\mathcal O_S\)-modules.  Hence
\begin{equation}
 \mathcal H^q(K_{-\infty})
 \simeq
 \operatorname*{colim}_{\lambda\to-\infty}
 \mathcal H^q(K_\lambda).
\label{eq:katz-oda-limiting-cohomology-colimit}
\end{equation}
Since \(\Omega_{S/\mathbb C}^1\) is locally free, tensoring with it
commutes with this colimit.  Naturality expressed by the upper square
of \eqref{eq:abstract-shifted-katz-oda-compatibility} therefore makes the
operators \(\nabla_\lambda^q\) descend uniquely to
\(\nabla_{-\infty}^q\), and gives the lower square of
\eqref{eq:abstract-shifted-katz-oda-compatibility}.  In the
colimit, the classes of
\(v\) and \(\mathcal H^q(\theta_{\lambda,\lambda-1})(v)\) agree.
Thus \eqref{eq:abstract-shifted-katz-oda-Leibniz} descends to
\eqref{eq:katz-oda-limiting-Leibniz}, proving that the resulting
operator is an algebraic connection.

For integrability, apply derived direct image to the full
\(L\)-filtration at each level \(\lambda\).  The projection formula
and \eqref{eq:katz-oda-associated-graded} identify its
\(r\)-th graded direct image with
\[
 \Omega_{S/\mathbb C}^r\otimes K_{\lambda-r}[-r].
\]
The column-zero first differential is
\(\nabla_{\lambda-r}^q\).  Consequently
\eqref{eq:abstract-L-spectral-sequence-Leibniz} gives
\eqref{eq:abstract-shifted-L-spectral-sequence-differential}; the
transition in its first term is forced by the target
\(\Omega_{S/\mathbb C}^{r+1}\otimes
\mathcal H^q(K_{\lambda-r-1})\).

Now take the homotopy colimit over the directed system of levels.
Translation by \(-r\) is cofinal in the indexing system, so
\[
 \operatorname*{hocolim}_{\lambda\to-\infty}K_{\lambda-r}
 \simeq K_{-\infty}
 \qquad(r\geq0).
\]
Using \eqref{eq:katz-oda-limiting-cohomology-colimit}, the first
differentials therefore induce
\[
 d_1(\eta\otimes v)
 =d\eta\otimes v+(-1)^r\eta\wedge\nabla_{-\infty}^q(v)
\]
on
\(\Omega_{S/\mathbb C}^\bullet\otimes
\mathcal H^q(K_{-\infty})\).  Their squares vanish because they are
first differentials of spectral sequences.  In base degree zero this
says precisely that the curvature of \(\nabla_{-\infty}^q\) vanishes.

Finally, suppose that the transition maps are isomorphisms below
\(\lambda_0\).  Then every sufficiently low cocone morphism
\(K_\lambda\to K_{-\infty}\) is an isomorphism.  The underlying
relative data form an indexed system of bounded \(S\)-linear
complexes with coherent \(S\)-flat terms by
\cref{def:katz-oda-compatible-filtered-de-rham-dg-module}, so
\cref{thm:perfect-direct-images-s-linear-complexes} makes
every \(K_\lambda\) perfect.  Identifying \(K_{-\infty}\) with any
sufficiently low \(K_\lambda\) proves its perfectness.
\end{proof}

\subsection{Finite-support systems and cohomological exactness}
\label{subsec:katz-oda-perfectness-exactness}

\begin{theorem}[Cohomological exactness for
Katz--Oda-compatible de Rham dg modules]
\label{thm:katz-oda-perfectness-exactness}
Let \((\mathcal A^\bullet,\mathbf D,F)\) be a relatively
\(S\)-flat, Katz--Oda-compatible filtered absolute de Rham dg module
on \(\mathfrak X/S\), and assume that \(S\) is connected.  Choose a
strictly
increasing sequence
\[
 \cdots<\lambda_{i-1}<\lambda_i<\lambda_{i+1}<\cdots
 \qquad(i\in\mathbb Z)
\]
in \(\Lambda\) which is cofinal toward \(-\infty\) and whose image in
\(\Lambda\) is stable under \(\lambda\mapsto\lambda-1\).  Form
\(\mathcal R_i\), \(\mathcal G_i\), and
\(\mathcal R_{-\infty}\) as in
\cref{lem:finite-support-s-linear-system}.  Cofinality identifies
\(\mathcal R_{-\infty}\) canonically with \(K_{-\infty}\) in
\eqref{eq:katz-oda-level-and-limiting-object}.

Assume:
\begin{enumerate}[label=\textup{(\roman*)}]
\item \(\mathcal G_i\) is acyclic for all but finitely many
  \(i\in\mathbb Z\);
\item for every geometric point \(s\to S\), every \(i\in\mathbb Z\),
  and every \(b\), the transition map
  \begin{equation}
   H^b(Li_s^*\mathcal R_{i+1})
   \longrightarrow
   H^b(Li_s^*\mathcal R_i)
  \label{eq:katz-oda-fibre-injectivity}
  \end{equation}
  is injective.
\end{enumerate}
Then the following conclusions hold.
\begin{enumerate}[label=\textup{(\alph*)}]
\item Every \(\mathcal H^b(\mathcal R_{-\infty})\),
  \(\mathcal H^b(\mathcal R_i)\), and
  \(\mathcal H^b(\mathcal G_i)\) is finite locally free.  Under the
  derived base-change identifications of
  \cref{lem:finite-support-s-linear-system}, formation of these
  cohomology sheaves commutes with arbitrary finite-type base change.
  Moreover,
  \(\mathcal H^b(\mathcal R_{-\infty})\) carries the canonical
  algebraic integrable connection \(\nabla_{-\infty}^b\) of
  \cref{prop:abstract-shifted-katz-oda}.

\item The transition triangles induce short exact sequences of vector
  bundles
  \begin{equation}
   0
   \longrightarrow\mathcal H^b(\mathcal R_{i+1})
   \longrightarrow\mathcal H^b(\mathcal R_i)
   \longrightarrow\mathcal H^b(\mathcal G_i)
   \longrightarrow0.
  \label{eq:katz-oda-short-exact}
  \end{equation}

\item Since \(\mathcal G_i\simeq0\) for all sufficiently large
  \(i\), the resulting transition isomorphisms define an eventual
  finite locally free high-index value
  \[
   \mathcal H_{+\infty}^b
   :=\mathcal H^b(\mathcal R_i),
   \qquad i\gg0,
  \]
  whose formation commutes with arbitrary finite-type base change.
  Every cocone morphism induces an inclusion
  \[
   \mathcal H^b(\mathcal R_i)
   \lhook\joinrel\longrightarrow
   \mathcal H^b(\mathcal R_{-\infty})
  \]
  of vector bundles with locally free quotient.  In particular, if
  \begin{equation}
   \mathcal I_{\lambda_i}^b
   :=\operatorname{Im}\!\left\{
    \mathcal H^b(\mathcal R_i)
    \longrightarrow
    \mathcal H^b(\mathcal R_{-\infty})
   \right\},
  \label{eq:abstract-filtered-image}
  \end{equation}
  then \(\mathcal I_{\lambda_i}^b\) is a subbundle with locally free
  quotient and
  \begin{equation}
   \nabla_{-\infty}^b(\mathcal I_{\lambda_i}^b)
   \subseteq
   \Omega_{S/\mathbb C}^1\otimes
   \mathcal I_{\lambda_i-1}^b.
  \label{eq:abstract-griffiths-transversality}
  \end{equation}
  Stability of the chosen sequence under translation by \(-1\) makes
  \(\mathcal I_{\lambda_i-1}^b\) one of the subbundles just defined.

\item All limiting-object, level, and transition-cone fibre dimensions are
  constant on \(S\).  At every geometric point \(s\to S\), there are
  canonical identifications
  \begin{align}
   \mathcal H^b(\mathcal R_i)\otimes\kappa(s)
   &\simeq
   H^b(Li_s^*\mathcal R_i)
   \simeq
   \mathbb H^b\!\left(
    \mathfrak X_s,
    Li_{s,\mathrm{sc}}^*\mathcal C^{\lambda_i,\bullet}
   \right),
  \label{eq:katz-oda-level-fibre}\\
   \mathcal H^b(\mathcal G_i)\otimes\kappa(s)
   &\simeq
   \operatorname{Coker}\!\left\{
    \mathbb H^b\!\left(
     \mathfrak X_s,
     Li_{s,\mathrm{sc}}^*\mathcal C^{\lambda_{i+1},\bullet}
    \right)
    \longrightarrow
    \mathbb H^b\!\left(
     \mathfrak X_s,
     Li_{s,\mathrm{sc}}^*\mathcal C^{\lambda_i,\bullet}
    \right)
   \right\}.
  \label{eq:katz-oda-cone-fibre}
  \end{align}
\end{enumerate}
\end{theorem}

\begin{proof}
After forgetting the action of positive-degree relative forms, the
induced relative system is an indexed system of bounded \(S\)-linear
complexes with coherent \(S\)-flat terms by
\cref{def:katz-oda-compatible-filtered-de-rham-dg-module}.
Hypothesis \textup{(i)} therefore makes
\cref{lem:finite-support-s-linear-system} applicable.  Choose
\(i_-\leq i_+\) such that \(\mathcal G_i\simeq0\) outside
\([i_-,i_+]\) and such that
\(\jmath_{i_-}\colon\mathcal R_{i_-}\to
\mathcal R_{-\infty}\) is an isomorphism.

By cofinality of the chosen sequence,
\cref{prop:abstract-shifted-katz-oda} gives every
\(\mathcal H^b(\mathcal R_{-\infty})\) its canonical algebraic
integrable connection.  Perfectness makes these cohomology sheaves
coherent, so the classical local-freeness theorem for coherent
integrable connections
\cite[Proposition~8.8]{KatzNilpotentConnections} makes them finite
locally free.  Transport through
\eqref{eq:katz-oda-low-representative} gives the same
conclusion for \(\mathcal H^b(\mathcal R_{i_-})\).
Therefore,
for any finite-type \(g\colon T\to S\),
the hyper-\(\operatorname{Tor}\) spectral sequence
\[
 E_2^{-p,q}
 =\mathcal{T}\!or_p^{g^{-1}\mathcal O_S}
 \left(g^{-1}\mathcal H^q(\mathcal R_{i_-}),\mathcal O_T\right)
 \Longrightarrow
 \mathcal H^{q-p}(Lg^*\mathcal R_{i_-})
\]
degenerates and yields canonical isomorphisms
\[
g^*\mathcal H^q(\mathcal R_{i_-})
\xrightarrow{\sim}
\mathcal H^q(Lg^*\mathcal R_{i_-}).
\]
In particular, at any geometric point \(i_s\colon s\to S\),
\[
 \dim_{\kappa(s)}H^q(Li_s^*\mathcal R_{i_-})
 =
 \dim_{\kappa(s)}i_s^*\mathcal H^q(\mathcal R_{i_-})
 =
 \dim_{\kappa(s)}i_s^*\mathcal H^q(\mathcal R_{-\infty}).
\]
These dimensions are therefore the ranks of the vector bundles
\(\mathcal H^q(\mathcal R_{-\infty})\), hence are constant because
\(S\) is connected.

Hypothesis \textup{(ii)} in degrees \(b\) and \(b+1\) breaks the
cohomology long exact sequence of every fibrewise transition triangle
into
\begin{equation}
 0\longrightarrow
 H^b(Li_s^*\mathcal R_{i+1})
 \longrightarrow H^b(Li_s^*\mathcal R_i)
 \longrightarrow H^b(Li_s^*\mathcal G_i)
 \longrightarrow0.
\label{eq:katz-oda-fibre-short-exact}
\end{equation}
Put
\[
 \beta_{i,b}(s):=\dim_{\kappa(s)}H^b(Li_s^*\mathcal G_i),
 \qquad
 \gamma_b(s):=\dim_{\kappa(s)}H^b(Li_s^*\mathcal R_{i_++1}).
\]
Telescoping the dimensions in
\eqref{eq:katz-oda-fibre-short-exact} gives
\begin{equation}
 \gamma_b(s)+\sum_{i=i_-}^{i_+}\beta_{i,b}(s)
 =\dim_{\kappa(s)}H^b(Li_s^*\mathcal R_{i_-}).
\label{eq:katz-oda-semicontinuity-sum}
\end{equation}
The right-hand side is constant.  The finitely many functions on the
left are upper semicontinuous by
\cref{lem:part-iii-constant-perfect-fibre-cohomology}.  Each is also
the constant right-hand side minus a sum of upper semicontinuous
functions, hence is lower semicontinuous.  Therefore they are locally
constant, hence
constant because \(S\) is connected.

Apply \cref{lem:part-iii-constant-perfect-fibre-cohomology} to
\(\mathcal R_{i_++1}\) and to every nonzero \(\mathcal G_i\).  Their
cohomology sheaves are finite locally free and commute with arbitrary
cohomology base change.  Moreover,
\[
 \dim_{\kappa(s)}H^b(Li_s^*\mathcal R_i)
 =\gamma_b(s)+\sum_{j=i}^{i_+}\beta_{j,b}(s)
 \qquad(i_-\leq i\leq i_++1),
\]
so the same perfect-complex criterion applies to every intervening
\(\mathcal R_i\).  The transition isomorphisms outside
\([i_-,i_+]\) extend these
conclusions to all remaining indices and define
\(\mathcal H_{+\infty}^b\).  This proves the local-freeness,
cohomology-base-change, and constancy assertions in \textup{(a)},
\textup{(c)}, and \textup{(d)}.

The connecting morphism
\[
 \mathcal H^b(\mathcal G_i)
 \longrightarrow\mathcal H^{b+1}(\mathcal R_{i+1})
\]
is a morphism of vector bundles compatible with base change.  Its
restriction to every geometric fibre is zero by
\eqref{eq:katz-oda-fibre-short-exact}, so it is zero.
The cohomology sequence of the transition triangle therefore breaks
into \eqref{eq:katz-oda-short-exact}.  For
\(i\geq i_-\), the composite from \(\mathcal R_i\) to
\(\mathcal R_{i_-}\), followed by \(\jmath_{i_-}\), consequently
induces an inclusion of vector bundles.  Its quotient has a finite
filtration with successive quotients among the
\(\mathcal H^b(\mathcal G_j)\), and is locally free.  For \(i<i_-\),
the cocone morphism is an isomorphism.  This proves \textup{(b)} and
the subbundle assertions in \textup{(c)}.

Combining ordinary cohomology base change with the derived geometric
fibre formula gives
\eqref{eq:katz-oda-level-fibre}; the
fibrewise short exact sequence gives
\eqref{eq:katz-oda-cone-fibre}.  Finally,
the lower square of
\eqref{eq:abstract-shifted-katz-oda-compatibility} sends the image of
the \(\lambda_i\)-level into the image of the
\((\lambda_i-1)\)-level.  This proves
\eqref{eq:abstract-griffiths-transversality} and completes the proof.
\end{proof}
 \section{Families with Fixed Pole Orders and Irregular Hodge Bundles}
\label{sec:families-fixed-pole-orders}
\label{sec:perfect-direct-images-base-change-katz-oda}
\label{sec:constancy-irregular-hodge-numbers}

This section applies the preceding dg-module formalism to a smooth
family of rational Landau--Ginzburg compactifications.  The primary
object is a single filtered \emph{absolute} logarithmic complex.  Its
degree-zero quotient for the Katz--Oda filtration produces the
relative Yu-filtered dg module.

\subsection{Smooth families with fixed pole orders and Yu complexes}
\label{subsec:fixed-pole-orders-yu-complexes}
\label{subsec:fixed-pole-absolute-yu-dg-module}

\begin{definition}[Smooth family with fixed pole orders]
\label{ass:fixed-pole-orders-dm-family}
\label{ass:fixed-pole-dm-family}
A \emph{smooth compactified Deligne--Mumford Landau--Ginzburg family
with fixed pole orders} consists of the following data and conditions.
\begin{enumerate}[label=\textup{(\roman*)}]
\item The base \(S\) is a smooth connected finite-type complex
  scheme, and
  \[
   \pi\colon\mathfrak X\longrightarrow S
  \]
  is a proper smooth morphism of pure relative dimension \(n\), where
  \(\mathfrak X\) is a separated tame Deligne--Mumford stack.
\item There are finitely many labelled relative effective Cartier
  divisors \(\mathfrak D_i\subset\mathfrak X\).  Their union
  \(\mathfrak D=\bigcup_i\mathfrak D_i\) is relative SNC: every
  intersection \(\mathfrak D_I\) is smooth over \(S\) of relative
  codimension \(|I|\), or is empty.  Put
  \(\mathfrak U=\mathfrak X\setminus\mathfrak D\).
\item There is a prescribed effective integral relative polar divisor
  \[
   \mathfrak P=\sum_i e_i\mathfrak D_i,
   \qquad e_i\in\mathbb Z_{\geq0},
  \]
  whose labelled coefficients are independent of the point of \(S\),
  together with a rational function \(\mathcal W\) on
  \(\mathfrak X\) that is regular on \(\mathfrak U\), has pole order
  bounded by \(\mathfrak P\), and satisfies
  \begin{equation}
   \operatorname{div}_{\infty}(\mathcal W_s)=\mathfrak P_s
   \qquad\text{for every geometric point }s\to S.
  \label{eq:fixed-pole-orders-fibrewise}
  \end{equation}
  On a component on which \(\mathcal W_s=0\), we use the convention
  \(\operatorname{div}_{\infty}(\mathcal W_s)=0\).
  Equivalently, if
  \(s_{\mathfrak P}\) denotes the canonical section of
  \(\mathcal O_{\mathfrak X}(\mathfrak P)\), the datum includes the
  section
  \[
   s_{\mathcal W}:=s_{\mathfrak P}\mathcal W
   \in H^0\!\left(\mathfrak X,
                   \mathcal O_{\mathfrak X}(\mathfrak P)\right).
  \]
  It is this section, together with \(s_{\mathfrak P}\), that is
  pulled back under an arbitrary change of base.
\item Every geometric fibre \(\mathfrak X_s\) has projective coarse
  moduli space, and
  \((\mathfrak X_s,\mathfrak D_s,\mathcal W_s)\) is an NC rational
  compactification of
  \((\mathfrak U_s,\mathcal W_s)\) in the sense of
  \cref{def:good-rational-stack-compactification}.
\end{enumerate}
\end{definition}
\begin{remark}[The product situation]
\label{rem:fixed-pair-product-family}
The principal applications below arise when the compactified space,
boundary, and polar divisor are constant over the base: there are
fixed data \((\mathscr X,\mathscr D,\mathscr P)\) and an isomorphism
over \(S\)
\begin{equation}
 (\mathfrak X,\mathfrak D,\mathfrak P)
 \simeq
 (\mathscr X,\mathscr D,\mathscr P)\times S.
\label{eq:split-fixed-pole-family}
\end{equation}
Then \(\mathfrak U\simeq\mathscr U\times S\), where
\(\mathscr U=\mathscr X\setminus\mathscr D\), and only the potential
varies.  This important special case is already covered by
\cref{ass:fixed-pole-orders-dm-family}.
\end{remark}

Because \(S\) and \(\mathfrak X/S\) are smooth and
\(\mathfrak D/S\) is relative SNC, \(\mathfrak X\) is smooth over
\(\mathbb C\), \(\mathfrak D\) is SNC over \(\mathbb C\), and the
logarithmic cotangent sequence is locally split exact:
\begin{equation}
 0\longrightarrow
 \pi^*\Omega_{S/\mathbb C}^1
 \longrightarrow
 \Omega_{\mathfrak X/\mathbb C}^1(\log\mathfrak D)
 \longrightarrow
 \Omega_{\mathfrak X/S}^1(\log\mathfrak D)
 \longrightarrow0.
\label{eq:part-iii-logarithmic-cotangent-transitivity}
\end{equation}
Put
\[
 \mathbf D
 :=d_{\mathfrak X/\mathbb C}
   +d_{\mathfrak X/\mathbb C}\mathcal W\wedge
\]
and consider the bounded absolute dg module
\begin{equation}
 \mathcal A_S^\bullet
 :=\left(
  \Omega_{\mathfrak X/\mathbb C}^\bullet(\log\mathfrak D)
  (*\mathfrak P),\mathbf D
 \right)
\label{eq:part-iii-absolute-yu-filtered-total-object}
\end{equation}
over \((\Omega_{\mathfrak X/\mathbb C}^\bullet,d)\), with action by
exterior multiplication.  Indeed, for an absolute \(r\)-form \(\eta\),
\begin{equation}
 \mathbf D(\eta\wedge\omega)
 =d\eta\wedge\omega+(-1)^r\eta\wedge\mathbf D\omega.
\label{eq:absolute-yu-dg-module-Leibniz}
\end{equation}

For \(\lambda\in\mathbb Q\), define a decreasing filtration by
subcomplexes by
\begin{equation}
 F^\lambda\mathcal A_S^m
 :=
 \begin{cases}
  0,&m<\lceil\lambda\rceil,\\[3pt]
  \Omega_{\mathfrak X/\mathbb C}^m(\log\mathfrak D)
  \bigl(\lfloor(m-\lambda)\mathfrak P\rfloor\bigr),
  &m\geq\lceil\lambda\rceil.
 \end{cases}
\label{eq:part-iii-absolute-yu-lattice}
\end{equation}
We abbreviate
\begin{equation}
\mathcal A_S^{\lambda,\bullet}
 :=F^\lambda\mathcal A_S^\bullet.
\label{eq:part-iii-absolute-yu-filtered-level-abbreviation}
\end{equation}
We call \(\mathcal A_S^{\lambda,\bullet}\) the absolute Yu lattice at
level \(\lambda\); these lattices form the Yu filtration \(F\) of the
absolute dg module.

Let \(L\) be the Katz--Oda filtration, equivalently the filtration
by the number of base forms, induced by the canonical dg-algebra morphism
\[
 \pi^{-1}(\Omega_{S/\mathbb C}^\bullet,d)
 \longrightarrow
 (\Omega_{\mathfrak X/\mathbb C}^\bullet,d),
\]
as in \eqref{eq:abstract-induced-base-degree-filtration}.
In the present logarithmic situation,
by \eqref{eq:part-iii-logarithmic-cotangent-transitivity} we have
\begin{align}
 \mathcal C_S^\bullet
 &:=\operatorname{gr}_L^0\mathcal A_S^\bullet\\
 &=\left(
  \Omega_{\mathfrak X/S}^\bullet(\log\mathfrak D)(*\mathfrak P),
  d_{\mathfrak X/S}+d_{\mathfrak X/S}\mathcal W\wedge
 \right),
\label{eq:part-iii-universal-relative-ambient-complex}\\
 \mathcal C_S^{\lambda,a}
 &:=\operatorname{gr}_L^0\mathcal A_S^{\lambda,\bullet}
   =F^\lambda\mathcal C_S^\bullet
\label{eq:part-iii-relative-yu-level-abbreviation}\\
 &=
 \begin{cases}
  0,&a<\lceil\lambda\rceil,\\[3pt]
  \Omega_{\mathfrak X/S}^a(\log\mathfrak D)
  \bigl(\lfloor(a-\lambda)\mathfrak P\rfloor\bigr),
  &a\geq\lceil\lambda\rceil.
 \end{cases}
\label{eq:part-iii-relative-yu-lattice}
\end{align}
We call \(\mathcal C_S^{\lambda,\bullet}\) the relative Yu lattice at
level \(\lambda\); the collection of these lattices is the relative
Yu filtration of \(\mathcal C_S^\bullet\).
The absolute-to-relative dg quotient of
\cref{lem:absolute-to-relative-dg-quotient} canonically makes these
complexes dg modules over
\((\Omega_{\mathfrak X/S}^\bullet,d_{\mathfrak X/S})\).

\begin{proposition}[Absolute and relative Yu complexes]
\label{prop:smooth-family-supplies-flat-yu-package}
\label{prop:yu-katz-oda-compatible-de-rham-module}
For every smooth family with fixed pole orders, the following
statements hold.
\begin{enumerate}[label=\textup{(\roman*)}]
\item The differential \(\mathbf D\) preserves every
  \(F^\lambda\mathcal A_S^\bullet\).  The filtration \(F\) is
  decreasing and exhaustive, and exterior multiplication induces
  \begin{equation}
   \Omega_{\mathfrak X/\mathbb C}^r\otimes
   F^{\lambda-r}\mathcal A_S^\bullet[-r]
   \longrightarrow F^\lambda\mathcal A_S^\bullet.
  \label{eq:part-iii-yu-filtered-absolute-de-rham-action}
  \end{equation}
\item The filtration \(L\) on each Yu lattice is the
  exterior-power filtration induced by
  \eqref{eq:part-iii-logarithmic-cotangent-transitivity}.  There are
  canonical isomorphisms of relative de Rham dg modules
  \begin{equation}
   \operatorname{gr}_L^r\mathcal A_S^{\lambda,\bullet}
   \xrightarrow{\ \sim\ }
   \pi^{-1}\Omega_{S/\mathbb C}^r
   \otimes_{\pi^{-1}\mathcal O_S}
   \mathcal C_S^{\lambda-r,\bullet}[-r]
   \qquad(r\geq0).
  \label{eq:part-iii-absolute-lattice-associated-graded}
  \end{equation}
\item Each \(\mathcal C_S^{\lambda,\bullet}\) is a bounded
  relative de Rham dg module whose terms are vector bundles on
  \(\mathfrak X\), hence coherent and flat over \(S\).  The
  inclusions for \(\mu\geq\lambda\),
  \[
   \mathcal C_S^{\mu,\bullet}
   \longrightarrow\mathcal C_S^{\lambda,\bullet},
  \]
  are relative dg-module morphisms and form a decreasing exhaustive
  filtration with locally finite lattice jumps.
\item Consequently
  \((\mathcal A_S^\bullet,\mathbf D,F)\) is a relatively
  \(S\)-flat, Katz--Oda-compatible filtered absolute de Rham dg
  module in the sense of
  \cref{def:katz-oda-compatible-filtered-de-rham-dg-module}, and its
  induced relative filtered dg module is exactly the system
  \(\{\mathcal C_S^{\lambda,\bullet}\}_{\lambda\in\mathbb Q}\)
  in \eqref{eq:part-iii-relative-yu-lattice}.
\end{enumerate}
\end{proposition}

\begin{proof}
The assertions are local for the smooth topology and may be checked
after pullback to an \(\acute{e}\)tale scheme atlas.  In relative SNC
coordinates, write
\[
 \mathfrak D=(x_1\cdots x_r=0),
 \qquad
 \mathcal W=x_1^{-e_1}\cdots x_r^{-e_r}G.
\]
Absolute
differentiation gives
\begin{equation}
 d_{\mathfrak X/\mathbb C}\mathcal W
 \in
 \Omega_{\mathfrak X/\mathbb C}^1(\log\mathfrak D)(\mathfrak P).
\label{eq:part-iii-absolute-dW-logarithmic-pole-bound}
\end{equation}
Indeed, differentiating the polar monomial produces
\(dx_i/x_i\), while differentiating \(G\) introduces no additional
Cartier pole.  Logarithmic differentiation preserves every integral
Cartier twist, and
\[
 \lfloor(m-\lambda)\mathfrak P\rfloor+\mathfrak P
 =\lfloor(m+1-\lambda)\mathfrak P\rfloor.
\]
It follows that \(\mathbf D\) preserves
\eqref{eq:part-iii-absolute-yu-lattice}.  Formula
\eqref{eq:absolute-yu-dg-module-Leibniz} proves the dg-module
compatibility.  If an \(r\)-form acts on a term of form degree
\(m-r\), then
\[
 (m-r)-(\lambda-r)=m-\lambda,
 \qquad
 m-r\geq\lceil\lambda-r\rceil
 \Longleftrightarrow m\geq\lceil\lambda\rceil,
\]
which proves the filtration shift in
\eqref{eq:part-iii-yu-filtered-absolute-de-rham-action}.
Exhaustiveness follows degreewise by allowing \(\lambda\) to tend to
\(-\infty\).

Taking exterior powers of
\eqref{eq:part-iii-logarithmic-cotangent-transitivity} gives
\[
 \operatorname{gr}_L^r
 \Omega_{\mathfrak X/\mathbb C}^m(\log\mathfrak D)
 \simeq
 \pi^*\Omega_{S/\mathbb C}^r\otimes
 \Omega_{\mathfrak X/S}^{m-r}(\log\mathfrak D).
\]
With \(a=m-r\), one has
\[
 \lfloor(m-\lambda)\mathfrak P\rfloor
 =\lfloor(a-(\lambda-r))\mathfrak P\rfloor,
 \qquad
 m\geq\lceil\lambda\rceil
 \Longleftrightarrow
 a\geq\lceil\lambda-r\rceil.
\]
Modulo the next Katz--Oda step, \(\mathbf D\) becomes the relative
twisted differential; the shift \([-r]\) records its sign.  This
proves \eqref{eq:part-iii-absolute-lattice-associated-graded} and,
at \(r=0\), the formulas
\eqref{eq:part-iii-universal-relative-ambient-complex} and
\eqref{eq:part-iii-relative-yu-lattice}.

Relative logarithmic forms are locally free for a relative SNC
divisor, and every rounded Cartier twist is a line bundle.  Thus all
terms of every relative Yu lattice are vector bundles on
\(\mathfrak X\).  They are flat over \(S\) because
\(\mathfrak X\to S\) is flat.  The monotonicity of the ceiling and
rounded divisors supplies the transition inclusions.  In degree \(a\)
they change only when \(\lambda\) crosses an integer or a number
\(a-j/e_i\) with \(e_i>0\); hence the set of lattice jumps is locally
finite.
Relative \(S\)-flatness and Katz--Oda compatibility now follow,
respectively, from termwise flatness and
\eqref{eq:part-iii-absolute-lattice-associated-graded}; see
\cref{def:katz-oda-compatible-filtered-de-rham-dg-module}.
\end{proof}

\subsection{Perfect direct images and canonical fibres}
\label{subsec:perfect-direct-images-fixed-pole-orders}
\label{subsec:fixed-pole-scalar-extension-perfect-fibres}

\begin{definition}[The direct-image system]
\label{def:yu-direct-image-system}
For \(\lambda\in\mathbb Q\), put
\begin{equation}
 \mathcal R_S^\lambda
 :=R\pi_{*,S}\mathcal C_S^{\lambda,\bullet}
 \in D(\mathcal O_S).
\label{eq:part-iii-perfect-direct-image-level}
\end{equation}
Because the lattice jumps are locally finite, we may set
\begin{equation}
 \mathcal C_S^{>\lambda,\bullet}
 :=
 \bigcup_{\substack{\mu>\lambda\\\mu\in\mathbb Q}}
 \mathcal C_S^{\mu,\bullet};
\label{eq:part-iii-yu-lattice-above-lambda}
\end{equation}
and, for every \(\lambda\), choose \(\mu>\lambda\) sufficiently close
to \(\lambda\) so that
\(\mathcal C_S^{>\lambda,\bullet}
=\mathcal C_S^{\mu,\bullet}\).  Define the \emph{rational graded
direct-image object}
\begin{equation}
 \mathcal G_S^\lambda
 :=
 \operatorname{Cone}\!\left(
  R\pi_{*,S}\mathcal C_S^{>\lambda,\bullet}
  \longrightarrow
  \mathcal R_S^\lambda
 \right).
\label{eq:part-iii-rational-graded-direct-image}
\end{equation}
Finally, define
\begin{equation}
\mathcal R_S^{-\infty}
 :=\operatorname*{hocolim}_{\lambda\to-\infty}
 \mathcal R_S^\lambda.
\label{eq:part-iii-limiting-direct-image}
\end{equation}
We call this the \emph{limiting direct image} and denote its cocone
maps by
\begin{equation}
 \jmath_\lambda\colon
 \mathcal R_S^\lambda\longrightarrow\mathcal R_S^{-\infty}.
\label{eq:part-iii-limiting-cocone-map}
\end{equation}
The nonpositive integers are cofinal, so
\(\mathcal R_S^{-\infty}\) is equivalently the homotopy colimit of
\(\mathcal R_S^0\to\mathcal R_S^{-1}\to\cdots\).
\end{definition}

Let \(g\colon T\to S\) be a finite-type morphism and let
\(g_{\mathfrak X}\colon\mathfrak X_T\to\mathfrak X\) be the base
change.  Pull back the labelled divisors, the polar divisor, and the
pair of sections \((s_{\mathfrak P},s_{\mathcal W})\); write the
resulting relative Yu lattices as
\(\mathcal C_T^{\lambda,\bullet}\).

\begin{proposition}[Perfect direct images and canonical fibres]
\label{prop:part-iii-yu-scalar-perfect-canonical-fibres}
\label{cor:smooth-yu-level-perfect-direct-images}
\label{cor:part-iii-derived-fibres-perfect-levels}
\label{cor:part-iii-relative-level-canonical-fibres}
\label{lem:part-iii-rational-graded-perfect-fibres}
For every \(\lambda\in\mathbb Q\), the following statements hold.
\begin{enumerate}[label=\textup{(\roman*)}]
\item
  There is a canonical isomorphism of relative de Rham dg modules
  \begin{equation}
   Lg_{\mathrm{sc}}^*\mathcal C_S^{\lambda,\bullet}
   \xrightarrow{\ \sim\ }
   \mathcal C_T^{\lambda,\bullet}.
  \label{eq:part-iii-level-arbitrary-base-change}
  \end{equation}
  It intertwines the differentials, relative dg-module actions, and
  all transition maps, and is compatible with composition of base
  changes.

\item
  The object \(\mathcal R_S^\lambda\) is perfect, and there is a
  canonical isomorphism
  \begin{equation}
   Lg^*\mathcal R_S^\lambda
   \xrightarrow{\ \sim\ }
   R(\pi_T)_{*,T}\mathcal C_T^{\lambda,\bullet}.
  \label{eq:smooth-yu-level-perfect-base-change}
  \end{equation}
  These isomorphisms are compatible with composition and with every
  transition map.

\item
  For every geometric point \(s\to S\), scalar extension and proper
  base change give canonical identifications
  \begin{align}
   Li_{s,\mathrm{sc}}^*\mathcal C_S^{\lambda,\bullet}
   &\simeq
   F^\lambda K_{\mathfrak X_s,\mathcal W_s}^\bullet,
  \label{eq:part-iii-universal-level-derived-fibre}\\
   Li_s^*\mathcal R_S^\lambda
   &\simeq
   R\Gamma\!\left(\mathfrak X_s,
    F^\lambda K_{\mathfrak X_s,\mathcal W_s}^\bullet\right),
  \label{eq:part-iii-perfect-level-derived-fibre}\\
   H^k(Li_s^*\mathcal R_S^\lambda)
   &\simeq
   \mathbb H^k\!\left(\mathfrak X_s,
    F^\lambda K_{\mathfrak X_s,\mathcal W_s}^\bullet\right).
  \label{eq:part-iii-perfect-level-fibre-hypercohomology}
  \end{align}
  They carry derived fibres of transition maps to the maps induced by
  the corresponding inclusions of fibrewise Yu lattices.

\item
  For every geometric point \(s\to S\) and every fixed
  \(\alpha\in\mathbb Q\cap[0,1)\), the spectral sequence of the
  integer-indexed fibrewise Yu filtration
  \(F_{\alpha,s}^p:=F^{p-\alpha}
  K_{\mathfrak X_s,\mathcal W_s}^\bullet\) degenerates at \(E_1\).
  This degeneration
  implies that, for every
  \(k\), restriction to \(\mathfrak U_s\) induces an injection
  \begin{equation}
   \mathbb H^k\!\left(
    \mathfrak X_s,F^\lambda
    K_{\mathfrak X_s,\mathcal W_s}^\bullet
   \right)
   \lhook\joinrel\longrightarrow
   H_{\mathrm{dR}}^k(\mathfrak U_s,\mathcal W_s).
  \label{eq:part-iii-fibre-level-injection}
  \end{equation}
  The compactification-comparison theorem of the first companion
  paper identifies the image of
  \eqref{eq:part-iii-fibre-level-injection} with
  \[
   F_{\mathrm{irr}}^\lambda
   H_{\mathrm{dR}}^k(\mathfrak U_s,\mathcal W_s).
  \]
  This is the only item of the proposition that invokes the
  \(E_1\)-degeneration theorem; the final identification in
  \textup{(v)} follows from the injections at the two adjacent Yu
  lattices.

\item
  Scalar extension identifies the lattices
  \(\mathcal C_S^{>\lambda,\bullet}\), and the rational graded
  direct-image objects are perfect and satisfy canonical
  finite-type derived base change:
  \begin{align}
   Lg_{\mathrm{sc}}^*\mathcal C_S^{>\lambda,\bullet}
   &\xrightarrow{\ \sim\ }
   \mathcal C_T^{>\lambda,\bullet},
  \label{eq:part-iii-above-lambda-base-change}\\
   Lg^*\mathcal G_S^\lambda
   &\xrightarrow{\ \sim\ }
   \mathcal G_T^\lambda.
  \label{eq:part-iii-rational-graded-base-change}
  \end{align}
  For a geometric point \(s\to S\), put
  \[
   \mathcal Q_{s,\lambda}^\bullet
   :=
   \mathcal C_s^{\lambda,\bullet}/
   \mathcal C_s^{>\lambda,\bullet}.
  \]
  Then there are canonical identifications
  \begin{align}
   Li_s^*\mathcal G_S^\lambda
   &\simeq
   R\Gamma(\mathfrak X_s,\mathcal Q_{s,\lambda}^\bullet),
  \label{eq:part-iii-graded-object-fibre-quotient-complex}\\
   H^q\!\left(Li_s^*\mathcal G_S^\lambda\right)
   &\simeq
   \Gr_{F_{\mathrm{irr}}}^{\lambda}
   H_{\mathrm{dR}}^q(\mathfrak U_s,\mathcal W_s).
  \label{eq:part-iii-graded-perfect-fibre-canonical-graded}
  \end{align}

\end{enumerate}
\end{proposition}

\begin{proof}
Smooth base change and the relative SNC condition give
\[
 g_{\mathfrak X}^*
 \Omega_{\mathfrak X/S}^a(\log\mathfrak D)
 \simeq
 \Omega_{\mathfrak X_T/T}^a(\log\mathfrak D_T).
\]
Pullback of labelled Cartier divisors commutes with coefficientwise
rounding because the coefficients \(e_i\) are fixed.  Hence pullback
identifies every term in
\eqref{eq:part-iii-relative-yu-lattice} with the
corresponding term over \(T\).  Functoriality of the universal
relative derivation identifies
\(d_{\mathfrak X/S}\mathcal W\) with
\(d_{\mathfrak X_T/T}\mathcal W_T\), so these termwise
identifications also intertwine the differentials.  The source terms
are \(S\)-flat, and therefore ordinary termwise pullback computes
derived scalar extension.  Functoriality of exterior multiplication
and of inclusions of Cartier twists proves all the compatibilities in
\textup{(i)}.

After forgetting positive-degree relative forms,
\cref{prop:smooth-family-supplies-flat-yu-package} gives an indexed
system of bounded \(S\)-linear complexes with coherent \(S\)-flat
terms.  Applying
\cref{thm:perfect-direct-images-s-linear-complexes} and
then using \textup{(i)} proves \textup{(ii)}.  Taking \(T=s\) in
\textup{(i)} gives the first formula in \textup{(iii)} by
\eqref{eq:companion-rational-yu-lattice}; taking \(T=s\) in
\textup{(ii)} gives the second, and taking cohomology gives the
third.  Compatibility with transitions follows from
\textup{(i)--(ii)}.

It remains to prove \textup{(iv)}.  By
\cref{ass:fixed-pole-orders-dm-family}, every fibre is an NC rational
compactification with projective coarse space, and its open coarse
space is quasiprojective.  Item~\textup{(i)} of
\cref{prop:nc-rational-level-graded-comparison} gives
\eqref{eq:part-iii-fibre-level-injection} and identifies its image
with the canonical irregular Hodge level.  Its injectivity is the
\(E_1\)-degeneration input from the second companion paper, while the
identification of the image is the compactification-comparison input
from the first.

For \textup{(v)}, local finiteness of the lattice jumps allows us to
choose
\(\mu>\lambda\), sufficiently close to \(\lambda\), such that
\(\mathcal C_S^{>\lambda,\bullet}=\mathcal C_S^{\mu,\bullet}\).
The possible jumps depend only on the fixed polar multiplicities, so
the same equality holds after every base change.  Thus
\eqref{eq:part-iii-above-lambda-base-change} follows from
\textup{(i)}.  Applying \textup{(ii)} at \(\mu\) shows that the two
terms defining \(\mathcal G_S^\lambda\) are perfect and commute with
derived base change.  The same assertions for their cone, including
\eqref{eq:part-iii-rational-graded-base-change}, follow formally.

Taking the derived fibre of the defining cone and using the termwise
inclusion
\(\mathcal C_s^{>\lambda,\bullet}\subseteq
\mathcal C_s^{\lambda,\bullet}\) gives
\eqref{eq:part-iii-graded-object-fibre-quotient-complex}.  Item
\textup{(ii)} of
\cref{prop:nc-rational-level-graded-comparison} then identifies
its cohomology with the rational graded piece, proving
\eqref{eq:part-iii-graded-perfect-fibre-canonical-graded}.  In
particular, the argument using the images at \(\lambda\) and
\(>\lambda\) remains valid when the two levels have different
fractional parts.

\end{proof}

\subsection{The limiting direct image and its stabilization}
\label{subsec:limiting-direct-image-fixed-pole-orders}
\label{subsec:fixed-pole-finite-support}

\begin{proposition}[Stabilization and finitely many lattice jumps]
\label{prop:part-iii-nonpositive-transitions-isomorphisms}
\label{cor:part-iii-finite-indexing-perfect-filtration}
The following statements hold.
\begin{enumerate}[label=\textup{(\roman*)}]
\item For every \(\lambda\leq0\), the transition morphism
  \begin{equation}
   \tau_\lambda\colon
   \mathcal R_S^0\longrightarrow\mathcal R_S^\lambda
  \label{eq:part-iii-nonpositive-transition-map}
  \end{equation}
  is an isomorphism in \(D_{\mathrm{perf}}(S)\), compatibly with
  arbitrary finite-type base change.  Consequently, the cocone map
  \begin{equation}
   \jmath_0\colon
   \mathcal R_S^0\xrightarrow{\ \sim\ }
   \mathcal R_S^{-\infty}
  \label{eq:part-iii-level-zero-limiting-isomorphism}
  \end{equation}
  is an isomorphism.

\item For every geometric point \(s\to S\), there is a canonical
  quasi-isomorphism
  \begin{equation}
   Li_s^*\mathcal R_S^{-\infty}
   \xrightarrow{\ \sim\ }
   R\Gamma\!\left(
    \mathfrak U_s,
    (\Omega_{\mathfrak U_s}^\bullet,
     d+d\mathcal W_s\wedge)
   \right).
  \label{eq:part-iii-limiting-twisted-de-rham-fibre}
  \end{equation}

\item For every geometric point \(s\to S\) and every \(q\), the
  subspaces
  \[
   \operatorname{Im}\!\left\{
    \mathbb H^q(\mathfrak X_s,
      F^\lambda K_{\mathfrak X_s,\mathcal W_s}^\bullet)
    \longrightarrow
    H_{\mathrm{dR}}^q(\mathfrak U_s,\mathcal W_s)
   \right\}
  \]
  change at only finitely many \(\lambda\in[0,n]\).
\end{enumerate}
\end{proposition}

\begin{proof}
By items~\textup{(ii)--(iii)} of
\cref{prop:part-iii-yu-scalar-perfect-canonical-fibres}, the source
and target of \(\tau_\lambda\) are perfect, and its derived fibre is
the transition between the corresponding fibrewise Yu lattices.  For
every \(\nu\leq0\), item~\textup{(iii)} of
\cref{prop:nc-rational-level-graded-comparison} gives the
quasi-isomorphism
\begin{equation}
 R\Gamma\!\left(
  \mathfrak X_s,F^\nu K_{\mathfrak X_s,\mathcal W_s}^\bullet
 \right)
 \xrightarrow{\ \sim\ }
 R\Gamma\!\left(
  \mathfrak U_s,
  (\Omega_{\mathfrak U_s}^\bullet,
   d+d\mathcal W_s\wedge)
 \right).
\label{eq:part-iii-nonpositive-level-ambient-comparison}
\end{equation}

For \(\nu=0\) and \(\nu=\lambda\), the two maps in
\eqref{eq:part-iii-nonpositive-level-ambient-comparison} identify the
derived fibre of \(\tau_\lambda\) with an isomorphism.  Fibrewise
conservativity for perfect complexes,
\cref{lem:part-iii-perfect-fibrewise-isomorphism-criterion}, proves
that \(\tau_\lambda\) is an isomorphism.  Its compatibility with
finite-type base change follows from the transition-compatible
isomorphisms in item~\textup{(ii)} of
\cref{prop:part-iii-yu-scalar-perfect-canonical-fibres}.  Thus every
map \(\mathcal R_S^{-j}\to\mathcal R_S^{-j-1}\), \(j\geq0\), is an
isomorphism, so
\(\jmath_0\) is an isomorphism.  This proves \textup{(i)}.

For \textup{(ii)}, compose the level-zero case of item~\textup{(iii)}
of the scalar-extension proposition with
\eqref{eq:part-iii-nonpositive-level-ambient-comparison} and the
inverse of \(Li_s^*\jmath_0\).

For \textup{(iii)}, write
\(\mathfrak P=\sum_i e_i\mathfrak D_i\).  In degree \(a\), a Yu
lattice changes only when \(\lambda\) crosses an integer or when
\((a-\lambda)e_i\in\mathbb Z\) for some \(e_i>0\).  Since
\(0\leq a\leq n\), only finitely many such values occur in
\([0,n]\).  Thus the relative Yu filtration, and hence its
cohomological image filtration, changes at only finitely many values
in that interval.
Moreover, the cocone identity
\(\jmath_{\lambda}\circ\tau_{\lambda}=\jmath_0\) and
\textup{(i)} show that the fibrewise image is the whole cohomology
for \(\lambda\leq0\), whereas it is zero for \(\lambda>n\), because
no term of the relative de Rham complex then survives.
Item~\textup{(iv)} of the
scalar-extension proposition identifies these images, and all the
intermediate ones, with the corresponding canonical irregular Hodge
levels.
\end{proof}

The rational graded direct-image objects were defined in
\eqref{eq:part-iii-rational-graded-direct-image}.  For every
\(\lambda\in\mathbb Q\), their definition gives a distinguished
triangle
\begin{equation}
 R\pi_{*,S}\mathcal C_S^{>\lambda,\bullet}
 \longrightarrow
 \mathcal R_S^\lambda
 \longrightarrow
 \mathcal G_S^\lambda
 \xrightarrow{+1}.
\label{eq:part-iii-rational-graded-triangle}
\end{equation}

\subsection{Irregular Hodge bundles}
\label{subsec:part-iii-total-twisted-cohomology-bundle}
\label{subsec:part-iii-graded-bundles-filtered-constancy}

Put
\[
 \mathcal H_S^q:=\mathcal H^q(\mathcal R_S^{-\infty}),
 \qquad
 \mathcal E_\lambda^q:=\mathcal H^q(\mathcal G_S^\lambda),
\]
and let
\[
 \mathcal I_\lambda^q
 :=
 \operatorname{Im}\!\left\{
  \mathcal H^q(\mathcal R_S^\lambda)
  \longrightarrow\mathcal H_S^q
 \right\}.
\]

\begin{theorem}[Irregular Hodge bundles in families with fixed pole
orders]
\label{thm:part-iii-total-cohomology-vector-bundles}
\label{thm:part-iii-relative-graded-filtered-bundles}
\label{thm:fixed-pole-dm-coefficient-bundles}
\label{thm:smooth-fixed-pole-filtered-bundles}
\label{thm:irregular-hodge-bundles-fixed-pole-orders}
\label{cor:part-iii-constancy-irregular-hodge-numbers}
\label{prop:part-iii-connection-smooth-base-change}
For every \(q\in\mathbb Z\) and \(\lambda\in\mathbb Q\), the
following hold.
\begin{enumerate}[label=\textup{(\roman*)}]
\item The sheaf \(\mathcal H_S^q\) is finite locally free and carries
  the canonical integrable connection
  \begin{equation}
   \nabla_S^q\colon
   \mathcal H_S^q\longrightarrow
   \Omega_{S/\mathbb C}^1\otimes\mathcal H_S^q.
  \label{eq:part-iii-total-twisted-katz-oda-connection}
  \end{equation}
  Its formation commutes with arbitrary finite-type base change:
  \begin{equation}
   g^*\mathcal H_S^q
   \xrightarrow{\ \sim\ }
   \mathcal H^q(Lg^*\mathcal R_S^{-\infty}).
  \label{eq:part-iii-total-cohomology-arbitrary-base-change}
  \end{equation}
  At a geometric point \(s\to S\),
  \begin{equation}
   \mathcal H_S^q\otimes\kappa(s)
   \simeq
   H_{\mathrm{dR}}^q(\mathfrak U_s,\mathcal W_s).
  \label{eq:part-iii-total-bundle-geometric-fibre}
  \end{equation}

\item The sheaf
  \(\mathcal H^q(\mathcal R_S^\lambda)\) is finite locally free and
  commutes with arbitrary finite-type base change.  Its map to
  \(\mathcal H_S^q\) is injective, with image
  \(\mathcal I_\lambda^q\), and this image is a subbundle with locally
  free quotient.  Its geometric fibre is
  \[
   \mathcal I_\lambda^q\otimes\kappa(s)
   \simeq
   F_{\mathrm{irr}}^\lambda
   H_{\mathrm{dR}}^q(\mathfrak U_s,\mathcal W_s).
  \]

\item The sheaf \(\mathcal E_\lambda^q\) is finite locally free,
  commutes with arbitrary finite-type base change, and has geometric
  fibre
  \[
   \mathcal E_\lambda^q\otimes\kappa(s)
   \simeq
   \Gr_{F_{\mathrm{irr}}}^{\lambda}
   H_{\mathrm{dR}}^q(\mathfrak U_s,\mathcal W_s).
  \]
  The distinguished triangle
  \eqref{eq:part-iii-rational-graded-triangle} gives a short
  exact sequence of vector bundles
  \begin{equation}
   0\longrightarrow
   \mathcal H^q\!\left(
    R\pi_{*,S}\mathcal C_S^{>\lambda,\bullet}
   \right)
   \longrightarrow
   \mathcal H^q(\mathcal R_S^\lambda)
   \longrightarrow
   \mathcal E_\lambda^q
   \longrightarrow0.
  \label{eq:part-iii-relative-filtered-bundle-short-exact}
  \end{equation}
  If
  \(\mathcal C_S^{>\lambda,\bullet}
   =\mathcal C_S^{\lambda,\bullet}\), then
  \(\mathcal E_\lambda^q=0\).

\item The filtered subbundles satisfy Griffiths transversality:
  \begin{equation}
   \nabla_S^q(\mathcal I_\lambda^q)
   \subseteq
   \Omega_{S/\mathbb C}^1\otimes\mathcal I_{\lambda-1}^q.
  \label{eq:part-iii-relative-griffiths-transversality}
  \end{equation}
  If \(g\colon T\to S\) is smooth, the base-change identification
  \(g^*\mathcal H_S^q\simeq\mathcal H_T^q\) carries the connection
  of the base-changed family to
  \begin{equation}
   g^*\mathcal H_S^q
   \xrightarrow{g^*\nabla_S^q}
   g^*\Omega_{S/\mathbb C}^1\otimes g^*\mathcal H_S^q
   \xrightarrow{dg\otimes1}
   \Omega_{T/\mathbb C}^1\otimes\mathcal H_T^q.
  \label{eq:part-iii-connection-smooth-pullback}
  \end{equation}

\item All total, filtered, and rational graded dimensions are
  independent of \(s\in S\).  Equivalently, for
  \(\alpha\in\mathbb Q\cap[0,1)\) and \(p,q\in\mathbb Z\), the
  irregular Hodge numbers
  \begin{equation}
   h_\alpha^{p,q}(\mathfrak U_s,\mathcal W_s)
   :=
   \dim_{\mathbb C}
   \operatorname{gr}_{F_{\alpha,\mathrm{irr}}}^{p}
   H_{\mathrm{dR}}^{p+q}(\mathfrak U_s,\mathcal W_s),
   \qquad
   F_{\alpha,\mathrm{irr}}^p
   :=F_{\mathrm{irr}}^{p-\alpha},
  \label{eq:part-iii-irregular-hodge-number-convention}
  \end{equation}
  are independent of \(s\).
\end{enumerate}
Together with
\cref{prop:part-iii-yu-scalar-perfect-canonical-fibres}, these are all
the conclusions stated in
\cref{thm:introduction-fixed-pole-orders}.
\end{theorem}

\begin{proof}
Fix \(\lambda\in\mathbb Q\).  By local finiteness of the lattice
jumps, choose
\(\mu>\lambda\), sufficiently close to \(\lambda\), such that
\(\mathcal C_S^{>\lambda,\bullet}
=\mathcal C_S^{\mu,\bullet}\) and the interval
\((\lambda,\mu)\) contains no integer.  Increasingly enumerate the
discrete set
\[
 \mathbb Z\cup(\lambda+\mathbb Z)\cup(\mu+\mathbb Z).
\]
This sequence is cofinal toward \(-\infty\) and stable under
translation by \(-1\); moreover, \(\lambda\) and \(\mu\) are
consecutive terms.  The corresponding Yu lattices vanish for all
sufficiently large indices, while all transitions between
nonpositive levels are isomorphisms by
\cref{prop:part-iii-nonpositive-transitions-isomorphisms}.  Since the
sequence has only finitely many terms in \([0,n]\), all but finitely
many transition cones are acyclic.  Item~\textup{(iv)} of
\cref{prop:part-iii-yu-scalar-perfect-canonical-fibres} identifies
the fibrewise transitions with inclusions between canonical
irregular Hodge levels.  Together with the derived fibre
identifications in
\cref{lem:finite-support-s-linear-system}, this verifies
hypothesis~\textup{(ii)} of
\cref{thm:katz-oda-perfectness-exactness}; the preceding finite-support
argument verifies hypothesis~\textup{(i)}.

Apply that theorem once to the relatively \(S\)-flat,
Katz--Oda-compatible filtered absolute Yu dg module of
\cref{prop:yu-katz-oda-compatible-de-rham-module}.  It
gives the local freeness, ordinary cohomology base change, short exact
transition sequences, subbundle assertions, the canonical integrable
connection, and Griffiths transversality in \textup{(i)--(iv)}, apart
from the smooth-pullback assertion.  The total fibre identification is
\eqref{eq:part-iii-limiting-twisted-de-rham-fibre}; the filtered
fibre identification follows from item~\textup{(iv)} of
\cref{prop:part-iii-yu-scalar-perfect-canonical-fibres}; and the
graded fibre identification follows from item~\textup{(v)} of the
same proposition.

Because the chosen sequence contains \(\lambda\), \(\lambda-1\), and
the chosen \(\mu>\lambda\) satisfying
\(\mathcal C_S^{>\lambda,\bullet}
=\mathcal C_S^{\mu,\bullet}\), the preceding
conclusions give all assertions for the originally fixed
\(\lambda\), including
\eqref{eq:part-iii-relative-filtered-bundle-short-exact} and
Griffiths transversality.  Since \(\lambda\) was arbitrary, they hold
at every rational level.  If
\(\mathcal C_S^{>\lambda,\bullet}
=\mathcal C_S^{\lambda,\bullet}\), the two terms defining
\(\mathcal G_S^\lambda\) coincide, so this cone is acyclic.

It remains to check smooth pullback of the connection.  For smooth
\(g\), flatness turns the derived base-change isomorphism into
\(g^*\mathcal H_S^q\simeq\mathcal H_T^q\).  The natural morphisms of
absolute de Rham dg algebras, together with the pullback of
\((s_{\mathfrak P},s_{\mathcal W})\), give a morphism from the
pulled-back absolute Yu object to the object on
\(\mathfrak X_T/T\).  Its degree-zero quotient for the Katz--Oda
filtration is
\eqref{eq:part-iii-level-arbitrary-base-change}, while on base
one-forms it is \(dg\).  Functoriality of the connecting morphisms in
\cref{prop:abstract-shifted-katz-oda}, followed by passage to the
homotopy colimit, gives
\eqref{eq:part-iii-connection-smooth-pullback}.

Finally, the dimensions in \textup{(v)} are ranks of the vector
bundles just constructed and hence are constant because \(S\) is
connected.  The quotient defining
\(\operatorname{gr}_{F_{\alpha,\mathrm{irr}}}^{p}\) has a finite
filtration with successive quotients the rational graded pieces whose
indices lie in
\([p-\alpha,p+1-\alpha)\).  Its dimension is the sum of their
constant ranks, proving the last assertion.
\end{proof}

\begin{remark}[The \(D\)-module perspective]
\label{subsec:part-iii-d-module-interpretation-audit}
\label{rem:part-iii-three-d-module-levels}
\label{rem:part-iii-d-module-interpretation-filtered-gap}
\label{rem:part-iii-no-hidden-stacky-irregular-direct-image}
\label{rem:part-iii-stacky-input-not-eliminated}
For smooth varieties admitting a projective good compactification,
the theorem can alternatively be obtained from the strict filtered
direct image of the exponential \(D\)-module and its comparison with
the Yu filtration \cite{SabbahYu,SabbahIHT}.

The stacky proof in this paper instead uses the explicit absolute and
relative Yu dg modules.  It invokes no relative irregular filtered
\(D\)-module or filtered-direct-image strictness theorem on a stack.
Its only stacky irregular Hodge inputs are the two companion results
on each geometric fibre: identification with the canonical
filtration and \(E_1\)-degeneration.  Thus the relative \(D\)-module
package is bypassed, whereas these fibrewise inputs remain essential.
\end{remark}

\subsection{Inertia sectors and orbifold bundles}
\label{subsec:general-relative-inertia-orbifold-bundles}

We now apply the theorem for families with fixed pole orders to the
relative inertia stack.

Let
\[
 e\colon I_S\mathfrak X\longrightarrow\mathfrak X
\]
be the relative inertia stack.  Since \(S\) is a scheme,
\(I_S\mathfrak X=I\mathfrak X\).  The inertia stack carries the
tautological automorphism \(g_{\mathrm{univ}}\) of every object pulled
back along \(e\).  Write
\begin{equation}
 I_S\mathfrak X
 =\coprod_{\gamma\in\mathcal B^{\mathrm c}}
   \mathfrak X_\gamma
\label{eq:general-compactified-inertia-components}
\end{equation}
for its finitely many connected components, and retain in
\(\mathcal B\subseteq\mathcal B^{\mathrm c}\) precisely those for
which
\[
 \mathfrak U_\gamma
 :=\mathfrak X_\gamma\times_{\mathfrak X}\mathfrak U
\]
is nonempty.  We call \(\mathfrak U_\gamma\) the open inertia sector
indexed by \(\gamma\).  Let
\(e_\gamma\colon\mathfrak X_\gamma\to\mathfrak X\) be evaluation and
put
\[
 \mathfrak D_{i,\gamma}
 :=\mathfrak X_\gamma\times_{\mathfrak X}\mathfrak D_i,
 \qquad
 \mathfrak D_\gamma
 :=(\mathfrak X_\gamma\setminus\mathfrak U_\gamma)_{\mathrm{red}},
 \qquad
 \mathcal W_\gamma:=e_\gamma^*\mathcal W.
\]
Empty restrictions \(\mathfrak D_{i,\gamma}\) are omitted.

\begin{lemma}[Polar geometry of inertia sectors]
\label{lem:general-relative-inertia-geometry}
For every \(\gamma\in\mathcal B\), the following assertions hold.
\begin{enumerate}[label=\textup{(\roman*)}]
\item The morphism
  \(\pi_\gamma\colon\mathfrak X_\gamma\to S\) is proper and smooth,
  of pure relative dimension, and its geometric fibres have
  projective coarse moduli spaces.
\item The nonempty divisors \(\mathfrak D_{i,\gamma}\) form a
  relative labelled SNC presentation of \(\mathfrak D_\gamma\).
\item For every geometric point \(s\to S\), each connected component
  of \(\mathfrak X_{\gamma,s}\), equipped with the reduced boundary
  and the relatively prime zero and pole divisors in
  \(\operatorname{div}(\mathcal W_{\gamma,s})\), is an NC rational
  compactification of
  its open sector.  If the restricted potential is zero, both
  divisors are taken to be zero.
\end{enumerate}
\end{lemma}

\begin{proof}
Because \(\mathfrak X\) is separated Deligne--Mumford, its inertia
is finite over \(\mathfrak X\).  It is therefore proper over \(S\)
and has only finitely many connected components.  Smooth-locally on
\(\mathfrak X\), a component of the inertia is represented by a
fixed locus \(V^g\), where \(V\to S\) is smooth and \(g\) has finite
order.  In characteristic zero \(V^g\) is smooth, and its tangent
bundle is the invariant part of \(T_{V/\mathbb C}|_{V^g}\).  Taking
\(g\)-invariants in the relative tangent sequence is exact because
\(g\) is linearly reductive, and gives
\[
 0\longrightarrow (T_{V/S}|_{V^g})^g
 \longrightarrow T_{V^g/\mathbb C}
 \longrightarrow T_{S/\mathbb C}|_{V^g}
 \longrightarrow0.
\]
The differential of \(V^g\to S\) is therefore surjective, so this
morphism, and hence \(I_S\mathfrak X\to S\), is smooth.  A nonempty
connected component
has open and closed image in the connected base \(S\), so it is
surjective and has pure relative dimension.  Finally,
\(\mathfrak X_{\gamma,s}\to\mathfrak X_s\) is finite.  Its coarse
space is finite over the projective coarse space of
\(\mathfrak X_s\), and is therefore projective.  This proves
\textup{(i)}.

We verify the boundary assertion on the same fixed-locus charts.
The tautological automorphism \(g_{\mathrm{univ}}\) preserves every labelled branch
\(\mathfrak D_i\) and acts on its normal line by a character.  If
that character were nontrivial at a point of
\(\mathfrak X_\gamma\cap\mathfrak D_i\), the local fixed locus would
be contained in \(\mathfrak D_i\).  Since
\(\mathfrak X_\gamma\) is connected and smooth, this would force
\(\mathfrak X_\gamma\subseteq\mathfrak D_i\), contrary to
\(\mathfrak U_\gamma\ne\varnothing\).  All boundary-normal
characters are therefore trivial.  Restricting relative SNC
coordinates to the fixed locus shows that every nonempty
\(\mathfrak D_{i,\gamma}\) is smooth over \(S\), and that every
intersection of \(r\) such divisors is empty or smooth over \(S\) of
relative codimension \(r\).  Their complement is
\(\mathfrak U_\gamma\), proving \textup{(ii)}.

For the last assertion, fix a geometric point \(s\to S\) and a
connected component of \(\mathfrak X_{\gamma,s}\).  The fibre
\((\mathfrak X_s,\mathfrak D_s,\mathcal W_s)\) is a strict NC
rational compactification in the sense of
\cite[Definition~2.3(a)]{WangCompactification}: split the finitely
many nonempty labelled divisors into their connected components if
necessary.  The
fixed-locus argument of
\cite[Lemma~3.16(i)]{WangCompactification} therefore applies and
gives precisely the claimed NC rational sector compactification.  Its
polar divisor is
\(\operatorname{div}_\infty(\mathcal W_{\gamma,s})\), which can be
strictly smaller than \(e_{\gamma,s}^*\mathfrak P_s\) when restriction
causes zero--pole cancellation.  This proves \textup{(iii)}.
\end{proof}

To apply the fixed-pole theorem sectorwise, impose the following
relative compactification hypothesis.

\begin{assumption}[Fixed-pole compactifications of inertia sectors]
\label{ass:sectorwise-fixed-pole-orders}
\label{ass:general-relative-inertia-fixed-pole}
For every \(\gamma\in\mathcal B\), the open family
\((\mathfrak U_\gamma,\mathcal W_\gamma)/S\) admits compactified data
\begin{equation}
 (\widetilde{\mathfrak X}_\gamma,
  \widetilde{\mathfrak D}_\gamma,
 \widetilde{\mathfrak P}_\gamma,
 \widetilde{\mathcal W}_\gamma)/S
\label{eq:general-sector-fixed-pole-orders-family}
\end{equation}
which form a smooth family with fixed pole orders and whose open part
is identified with
\((\mathfrak U_\gamma,\mathcal W_\gamma)\).
\end{assumption}

By \cref{lem:general-relative-inertia-geometry}, one may take
\(\widetilde{\mathfrak X}_\gamma=\mathfrak X_\gamma\) whenever the
divisors
\(\operatorname{div}_\infty(\mathcal W_{\gamma,s})\) assemble with
multiplicities independent of \(s\).  The assumption also holds if
\(\mathcal W\) extends to an \(S\)-morphism
\(\mathfrak X\to\mathbb P^1_S\) whose pole divisor is
\(\mathfrak P\), because its restriction to a nonempty open sector
is again morphic with a relative Cartier pole divisor.

The restriction of \(g_{\mathrm{univ}}\) acts on
\(T_{\mathfrak X/S}|_{\mathfrak X_\gamma}\).  On an \'etale chart,
write its eigenvalues as
\(\exp(2\pi\sqrt{-1}\,\theta_1),\ldots,
\exp(2\pi\sqrt{-1}\,\theta_n)\), with
\(0\leq\theta_j<1\).  The \(\theta_j\) are locally constant on the
connected component, and we define
\begin{equation}
 a_\gamma
 :=\operatorname{age}(\mathfrak U_\gamma)
 :=\sum_{j=1}^n\theta_j
 \in\mathbb Q_{\geq0}.
\label{eq:general-relative-sector-age}
\end{equation}
For \(k\in\mathbb Z\) and \(\mu\in\mathbb Q\), apply
\cref{thm:irregular-hodge-bundles-fixed-pole-orders} to
\eqref{eq:general-sector-fixed-pole-orders-family}, and denote the resulting
total, filtered, and rational graded bundles by
\[
 \mathcal H_\gamma^k,
 \qquad
 \mathcal I_{\gamma,\mu}^k,
 \qquad
 \mathcal E_{\gamma,\mu}^k.
\]

For \(r,\lambda\in\mathbb Q\), set
\begin{align}
 \mathcal H_{\mathrm{orb}}^r
 &:={}
 \bigoplus_{\substack{\gamma\in\mathcal B\\
              r-2a_\gamma\in\mathbb Z}}
 \mathcal H_\gamma^{\,r-2a_\gamma},
\label{eq:general-orbifold-total-bundle}\\
 \mathcal I_{\mathrm{orb},\lambda}^r
 &:={}
 \bigoplus_{\substack{\gamma\in\mathcal B\\
              r-2a_\gamma\in\mathbb Z}}
 \mathcal I_{\gamma,\lambda-a_\gamma}^{\,r-2a_\gamma},
\label{eq:general-orbifold-filtered-bundle}\\
 \mathcal E_{\mathrm{orb},\lambda}^r
 &:={}
 \bigoplus_{\substack{\gamma\in\mathcal B\\
              r-2a_\gamma\in\mathbb Z}}
 \mathcal E_{\gamma,\lambda-a_\gamma}^{\,r-2a_\gamma}.
\label{eq:general-orbifold-graded-bundle}
\end{align}
On a geometric fibre \(s\to S\), these conventions amount to
\begin{align*}
 H_{\mathrm{orb,dR}}^r(\mathfrak U_s,\mathcal W_s)
 &:={}
 \bigoplus_{\substack{\gamma\in\mathcal B\\
              r-2a_\gamma\in\mathbb Z}}
 H_{\mathrm{dR}}^{\,r-2a_\gamma}
 (\mathfrak U_{\gamma,s},\mathcal W_{\gamma,s}),\\
 F_{\mathrm{orb}}^\lambda
 H_{\mathrm{orb,dR}}^r(\mathfrak U_s,\mathcal W_s)
 &:={}
 \bigoplus_{\substack{\gamma\in\mathcal B\\
              r-2a_\gamma\in\mathbb Z}}
 F_{\mathrm{irr}}^{\,\lambda-a_\gamma}
 H_{\mathrm{dR}}^{\,r-2a_\gamma}
 (\mathfrak U_{\gamma,s},\mathcal W_{\gamma,s}).
\end{align*}

\begin{theorem}[Orbifold irregular Hodge bundles]
\label{thm:orbifold-bundles-fixed-pole-orders}
\label{thm:general-orbifold-fixed-pole-bundles}
Let a smooth family with fixed pole orders satisfy
\cref{ass:sectorwise-fixed-pole-orders}.  Then, for every
\(r,\lambda\in\mathbb Q\), the following assertions hold.
\begin{enumerate}[label=\textup{(\roman*)}]
\item The sheaf \(\mathcal H_{\mathrm{orb}}^r\) is finite locally
  free, and \(\mathcal I_{\mathrm{orb},\lambda}^r\) is a subbundle
  with locally free quotient.  The rational graded quotient is
  \(\mathcal E_{\mathrm{orb},\lambda}^r\).
\item Their fibres at \(s\to S\) are canonically
  \[
   H_{\mathrm{orb,dR}}^r(\mathfrak U_s,\mathcal W_s),
   \qquad
   F_{\mathrm{orb}}^\lambda
   H_{\mathrm{orb,dR}}^r(\mathfrak U_s,\mathcal W_s),
   \qquad
   \Gr_{F_{\mathrm{orb}}}^{\lambda}
   H_{\mathrm{orb,dR}}^r(\mathfrak U_s,\mathcal W_s).
  \]
  The bundles, inclusions, quotients, and these fibre
  identifications commute with arbitrary finite-type base change.
\item The direct-sum connection is algebraic and integrable and
  satisfies
  \begin{equation}
   \nabla_{\mathrm{orb}}^r
   (\mathcal I_{\mathrm{orb},\lambda}^r)
   \subseteq
   \Omega_{S/\mathbb C}^1\otimes
   \mathcal I_{\mathrm{orb},\lambda-1}^r.
  \label{eq:general-orbifold-griffiths-transversality}
  \end{equation}
  It is functorial under smooth base change.
\item The orbifold irregular Hodge numbers
  \begin{equation}
   h_{\mathrm{orb}}^{r,\lambda}
   (\mathfrak U_s,\mathcal W_s)
   :=\dim_{\mathbb C}
   \Gr_{F_{\mathrm{orb}}}^{\lambda}
   H_{\mathrm{orb,dR}}^r(\mathfrak U_s,\mathcal W_s)
  \label{eq:general-orbifold-irregular-hodge-numbers}
  \end{equation}
  are independent of \(s\in S\).  Equivalently, the Harder--Lee
  indexing
  \[
   f_{\mathrm{orb}}^{\lambda,\mu}
   :=\dim_{\mathbb C}
   \Gr_{F_{\mathrm{orb}}}^{\lambda}
   H_{\mathrm{orb,dR}}^{\lambda+\mu}
  \]
  is constant in the family.
\end{enumerate}
\end{theorem}

\begin{proof}
Apply \cref{thm:irregular-hodge-bundles-fixed-pole-orders} to the finitely
many sector families supplied by
\cref{ass:sectorwise-fixed-pole-orders}, and take the
age-shifted finite direct sums.  Finite direct sums preserve local
freeness, exactness, integrability, and arbitrary base change.  The
sectorwise rational graded quotients give
\eqref{eq:general-orbifold-graded-bundle}, and their fibre
identifications give the three fibre descriptions in \textup{(ii)}.

For transversality, the \(\gamma\)-summand satisfies
\[
 \nabla_\gamma
 (\mathcal I_{\gamma,\lambda-a_\gamma})
 \subseteq
 \Omega_{S/\mathbb C}^1\otimes
 \mathcal I_{\gamma,(\lambda-a_\gamma)-1}
 =
 \Omega_{S/\mathbb C}^1\otimes
 \mathcal I_{\gamma,(\lambda-1)-a_\gamma}.
\]
These inclusions assemble to
\eqref{eq:general-orbifold-griffiths-transversality}.  Inertia
commutes with base change; a connected component may split after
base change, but the sum of its cohomology bundles is the pullback of
the original sector bundle.  This proves all the asserted
compatibilities.  Finally,
\eqref{eq:general-orbifold-irregular-hodge-numbers} is the rank of
\(\mathcal E_{\mathrm{orb},\lambda}^r\), hence is constant because
\(S\) is connected.
\end{proof}
 \section{Toric Newton Families and Coefficient Invariance}
\label{sec:toric-newton-families}

This section supplies the toric application of the theorem for
families with fixed pole orders.

\subsection{Newton polytopes at infinity}
\label{subsec:toric-newton-polytopes-at-infinity}

We use the stacky-fan conventions of Borisov--Chen--Smith
\cite[Section~3]{BCS}.  Fix a smooth quasiprojective toric
Deligne--Mumford stack
\begin{equation}
 \mathscr U_\Gamma
 =
 \mathscr X_{(\mathbf L,\Sigma,\beta)}\times T_Q,
\label{eq:toric-fixed-open-stack}
\end{equation}
where \(\mathbf L\) is finitely generated,
\(L=\mathbf L/\mathbf L_{\mathrm{tor}}\), the simplicial fan
\(\Sigma\) is quasiprojective with rays spanning \(L_{\mathbb R}\),
the map
\(\beta\colon\mathbb Z^{\Sigma(1)}\to\mathbf L\) has finite
cokernel, \(Q\) is a free lattice, and
\(T_Q:=\operatorname{Spec}\mathbb C[Q^\vee]\).  Put
\begin{equation}
 N:=L\oplus Q,
 \qquad
 M:=N^\vee.
\label{eq:toric-character-lattices}
\end{equation}
Write \(T_N:=\operatorname{Spec}\mathbb C[M]\) for the coarse
dense torus.
The rational fan of \(\mathscr U_\Gamma\) is
\(\Gamma=\Sigma\times\{0\}\) in \(N_{\mathbb R}\).  Characters
are taken from the coarse dense torus; a finite generic gerbe does not
change \(M\).

Fix a lattice polytope \(P\subset M_{\mathbb R}\) such that
\begin{equation}
 0\in P,
 \qquad
 |\Gamma|\subseteq\operatorname{Cone}(P)^\vee.
\label{eq:toric-newton-support-condition}
\end{equation}
We refer to \eqref{eq:toric-newton-support-condition} as the
\emph{support condition} for \((\Gamma,P)\).
Thus every \(x^m\), \(m\in P\cap M\), is regular on
\(\mathscr U_\Gamma\).  Set
\begin{equation}
 V_P:=\bigoplus_{m\in P\cap M}\mathbb Cx^m,
 \qquad
 h=\sum_{m\in P\cap M}a_mx^m.
\label{eq:toric-coefficient-space}
\end{equation}
The polytope \(P\) is fixed before the coefficients; in particular,
the coefficient of \(x^0\) is allowed to vanish.  With \(A_m\) the
coordinate dual to \(x^m\), the universal potential is
\begin{equation}
 \mathcal W_P:=\sum_{m\in P\cap M}A_mx^m
 \quad\text{on}\quad
 \mathscr U_\Gamma\times V_P.
\label{eq:toric-universal-potential}
\end{equation}

For \(h\in V_P\), write
\begin{equation}
 \operatorname{Supp}(h)
 :=\{m\in P\cap M\mid a_m\ne0\},
 \qquad
 \Delta_\infty(h)
 :=\operatorname{Conv}\bigl(\{0\}\cup\operatorname{Supp}(h)\bigr).
\label{eq:toric-newton-polytope-at-infinity}
\end{equation}
Thus \(\Delta_\infty(h)\) is the Newton polytope of \(h\) at
infinity.  For a nonempty face \(F\preceq\Delta_\infty(h)\), choose
\(m_F\in F\cap M\) and put
\begin{equation}
 \begin{aligned}
  M_F&:=M\cap\operatorname{Span}_{\mathbb R}(F-F),
  &T_F&:=\operatorname{Spec}\mathbb C[M_F],\\
  h_F&:=\sum_{m\in F\cap M}a_mx^m,
  &\overline h_F&:=x^{-m_F}h_F.
 \end{aligned}
\label{eq:toric-effective-face-polynomial}
\end{equation}
The lattice \(M_F\) is saturated.  Multiplication by a character
shows that the smoothness of \(Z(\overline h_F)\subset T_F\) is
independent of \(m_F\).  Moreover, the smooth surjective torus map
\(T_N\to T_F\) gives an equivalent criterion on \(T_N\): there is no
point \(t\in T_N\) such that
\[
 h_F(t)=0,
 \qquad
 \partial_vh_F(t)=0\quad\text{for every }v\in N_\mathbb C,
\]
where \(\partial_v(x^m):=\langle v,m\rangle x^m\).

\begin{definition}[Newton nondegeneracy at infinity]
\label{def:toric-newton-nondegenerate-at-infinity}
The Laurent polynomial \(h\) is \emph{Newton nondegenerate at
infinity} if, for every face
\(F\preceq\Delta_\infty(h)\) with \(0\notin F\), the hypersurface
\[
 Z(\overline h_F)\subset T_F
\]
is empty or smooth of pure codimension one.  Define the fixed-Newton
coefficient locus
\begin{equation}
 \mathcal U_P
 :=\{h\in V_P\mid
       \Delta_\infty(h)=P
       \text{ and }h\text{ is Newton nondegenerate at infinity}\}.
\label{eq:toric-fixed-newton-nondegenerate-locus}
\end{equation}
\end{definition}

This is the standard definition of nondegeneracy with respect to the
Newton polytope at infinity; compare
\cite[\S4, equation~(23)]{Yu} and \cite[Introduction]{WangNewton}.

It is standard that \(\mathcal U_P\) is a nonempty Zariski-open dense
subset of \(V_P\).  The condition \(\Delta_\infty(h)=P\) says that
the coefficients at the vertices of \(P\) other than \(0\) do not
vanish.  On this vertex-open set, the complement of
\(\mathcal U_P\) is the finite union of the relevant
face-discriminant loci.  Applying the principal \(A\)-determinant to
the faces maximal among those avoiding \(0\) shows that this union is
proper and closed; see \cite[Chapter~10]{GKZ} and
\cite[Proposition~1.1]{SpVdB}.  In particular, \(\mathcal U_P\) is
smooth, irreducible, and connected.

\subsection{Toric compactifications with fixed pole orders}
\label{subsec:toric-fixed-pole-orders}
\label{subsec:toric-fixed-pole-package}

Define the lower support function
\begin{equation}
 \psi_P(u):=\min_{m\in P}\langle u,m\rangle,
 \qquad u\in N_{\mathbb R}.
\label{eq:toric-lower-support-function}
\end{equation}
It is an integral piecewise-linear lower-convex support function
(equivalently, a concave function in the usual analytic convention).
By
\eqref{eq:toric-newton-support-condition} and \(0\in P\),
\begin{equation}
 \left.\psi_P\right|_{|\Gamma|}=0.
\label{eq:toric-support-function-zero}
\end{equation}

\begin{proposition}[Adapted toric stack compactification]
\label{prop:toric-adapted-stack-compactification}
There is a smooth proper toric Deligne--Mumford stack
\(\overline{\mathscr X}_{\Gamma,P}\), with projective coarse moduli
space, whose rational fan \(\overline\Gamma_P\) has the following
properties:
\begin{enumerate}[label=\textup{(\roman*)}]
\item \(\mathscr U_\Gamma\) is the unchanged open toric substack,
  including its generic gerbe;
\item \(\Gamma\) is a full subfan of \(\overline\Gamma_P\);
\item \(\psi_P\) is linear on every cone of
  \(\overline\Gamma_P\).
\end{enumerate}
Consequently,
\[
 \mathscr D_{\Gamma,P}
 :=(\overline{\mathscr X}_{\Gamma,P}
       \setminus\mathscr U_\Gamma)_{\mathrm{red}}
\]
is a labelled SNC divisor.
\end{proposition}

\begin{proof}
Sumihiro's equivariant embedding theorem
\cite[Theorem~2.5]{Sumihiro} gives a torus-equivariant locally closed
immersion of the quasiprojective coarse toric variety into a
projective space with linear torus action.  Normalizing the closure
produces a complete projective fan \(\Delta^{\mathrm c}\) containing
\(\Gamma\) unchanged.  Refine \(\Delta^{\mathrm c}\) by the domains
of linearity of \(\psi_P\).  This common refinement does not subdivide
\(\Gamma\): if \(F_0\) is the smallest face of \(P\) containing
\(0\), write \(0\) as a strictly positive convex combination of the
vertices of \(F_0\).  For \(u\in|\Gamma|\), all their pairings with
\(u\) are nonnegative and the same convex combination is zero, so
they all vanish.  Every cone of \(\Gamma\) is consequently contained
in one domain of linearity of \(\psi_P\).

The common refinement remains projective.  If \(\phi\) is an integral
strictly convex support function for \(\Delta^{\mathrm c}\), then
\(\phi+\psi_P\) bends strictly across every wall: \(\phi\) supplies
the strict bending across an old wall, and \(\psi_P\) supplies it
across a new wall inside a cone of \(\Delta^{\mathrm c}\).  The wall
criterion and the ampleness criterion give projectivity; see
\cite[Lemma~6.1.13 and Theorem~6.1.14]{CLS}.

Perform the barycentric star subdivisions of all nonzero cones not
belonging to \(\Gamma\), in decreasing dimension, and leave
\(\Gamma\) untouched.  The resulting fan is simplicial, and every
cone outside \(\Gamma\) contains a new barycentric ray; hence
\(\Gamma\) is full.  Each subdivision is projective by
\cite[Proposition~11.1.6(c)]{CLS}, and \(\psi_P\) remains linear.

Retain the old stacky vectors.  On every new ray choose an element of
\(\operatorname{im}(\beta)\oplus Q\subset\mathbf L\oplus Q\) whose
image in \(N\) is a positive multiple of the primitive generator.
The Cox construction \cite[Section~3]{BCS} gives a smooth toric
Deligne--Mumford stack; on the Cox open belonging to \(\Gamma\), the
new coordinates are invertible and the inclusion of the new stacky
vectors in \(\operatorname{im}(\beta)\oplus Q\) identifies this open
with \(\mathscr U_\Gamma\), including its generic gerbe.  Fullness
makes the complement the union of the divisors indexed by the new
rays; they form a labelled SNC divisor on the smooth stack.
\end{proof}

For every ray \(\rho\in\overline\Gamma_P(1)\), let \(v_\rho\in N\)
be its primitive generator and write the image of its stacky vector
as
\[
 \overline b_\rho=c_\rho v_\rho,
 \qquad c_\rho\in\mathbb Z_{>0}.
\]
Let \(\mathscr D_\rho\) be the corresponding prime stack divisor and
put
\begin{equation}
 \nu_\rho(P):=\min_{m\in P}\langle v_\rho,m\rangle,
 \qquad
 e_\rho^{\mathrm{st}}(P):=-c_\rho\nu_\rho(P).
\label{eq:toric-stacky-pole-order}
\end{equation}
These are respectively an integer \(\leq0\) and an integer
\(\geq0\).  For every old ray, \(\nu_\rho(P)=0\).  Define the fixed
stacky Newton pole divisor
\begin{equation}
 \mathscr P_{\Gamma,P}
 :=
 \sum_{\rho\in\overline\Gamma_P(1)\setminus\Gamma(1)}
 e_\rho^{\mathrm{st}}(P)\mathscr D_\rho,
\label{eq:toric-fixed-stacky-pole-divisor}
\end{equation}
omitting the zero coefficients.
Thus the coefficient of \(\mathscr D_\rho\) in a Yu lattice of form
degree \(a\) and filtration level \(\lambda\) is
 \(\left\lfloor
  (a-\lambda)c_\rho\bigl(-\nu_\rho(P)\bigr)
 \right\rfloor\).

\begin{theorem}[Toric family with fixed pole orders]
\label{thm:toric-family-fixed-pole-orders}
\label{thm:toric-fixed-pole-family}
For every \(h\in\mathcal U_P\), the rational extension of \(h\) to
\(\overline{\mathscr X}_{\Gamma,P}\) has polar divisor exactly
\(\mathscr P_{\Gamma,P}\), and
\[
 (\overline{\mathscr X}_{\Gamma,P},
  \mathscr D_{\Gamma,P},h)
\]
is a projective labelled NC rational stack compactification of
\((\mathscr U_\Gamma,h)\).

Moreover, on
\[
 \mathfrak X_{\Gamma,P}
 :=\overline{\mathscr X}_{\Gamma,P}\times\mathcal U_P
\]
the universal potential has divisorial polar part
\[
 \mathfrak P_{\Gamma,P}
 :=\mathscr P_{\Gamma,P}\times\mathcal U_P.
\]
Thus the product family, with boundary
\(\mathscr D_{\Gamma,P}\times\mathcal U_P\), satisfies
\cref{ass:fixed-pole-orders-dm-family}.
\end{theorem}

\begin{proof}
The support condition makes \(h\) regular on
\(\mathscr U_\Gamma\).  On a Cox atlas, the character \(x^m\) is a
rational monomial whose exponent in the \(\rho\)-coordinate is
\(\langle\overline b_\rho,m\rangle\).  Hence
\begin{equation}
 \operatorname{ord}_{\mathscr D_\rho}(x^m)
 =c_\rho\langle v_\rho,m\rangle.
\label{eq:toric-stack-valuation-character}
\end{equation}
If \(\nu_\rho(P)<0\), its minimizing face
\[
 F_\rho:=\{m\in P\mid
             \langle v_\rho,m\rangle=\nu_\rho(P)\}
\]
does not contain \(0\).  It contains a vertex \(v\ne0\), and
\(\Delta_\infty(h)=P\) gives \(a_v\ne0\).  After factoring \(x^v\), the
initial coefficient is the nonzero Laurent polynomial
\(x^{-v}h_{F_\rho}\) on the dense torus of \(\mathscr D_\rho\), so
there is no cancellation at its generic point.  Formula
\eqref{eq:toric-stack-valuation-character} therefore gives
\[
 \operatorname{ord}_{\mathscr D_\rho}(h)
 =c_\rho\nu_\rho(P).
\]
When \(\nu_\rho(P)=0\), no pole occurs.  This proves that the polar
divisor is exactly \eqref{eq:toric-fixed-stacky-pole-divisor}.
The same initial-form argument with the coefficient functions
\(A_m\) proves the asserted equality on the universal family.

It remains to check the NC rational condition near the polar locus.
For a cone \(\sigma\in\overline\Gamma_P\), linearity of \(\psi_P\)
defines the face
\[
 F_\sigma
 :=\{m\in P\mid
       \langle u,m\rangle=\psi_P(u)
       \text{ for every }u\in\sigma\}.
\]
The orbit \(\mathscr O_\sigma\) lies in
\(|\mathscr P_{\Gamma,P}|\) exactly when \(0\notin F_\sigma\).
Choose \(m_\sigma\in F_\sigma\cap M\).  On the affine chart of
\(\sigma\), the normalized function \(x^{-m_\sigma}h\) is regular,
and its restriction to the orbit is, up to a unit, the pullback of
\(\overline h_{F_\sigma}\) along the smooth torus morphism
\[
 \operatorname{Spec}\mathbb C[M\cap\sigma^\perp]
 \longrightarrow T_{F_\sigma}.
\]
If the orbit is polar, Newton nondegeneracy at infinity makes the
corresponding zero hypersurface empty or smooth and reduced.  After
pullback to a smooth Cox atlas, its differential has a nonzero
component in an orbit direction at every point of its zero locus.
Together with the toric boundary coordinates, the orbit coordinates
form a regular system of parameters.  Hence the zero divisor is
smooth near the orbit and meets every boundary stratum transversely.
This is the
standard toric Newton argument of
\cite[Proposition~4.3]{Yu}; the smooth torus map handles a
lower-dimensional face, while the Cox atlas transfers the criterion
to the Deligne--Mumford stack.  Thus the required zero--pole condition
holds in a neighbourhood of the polar locus.

The remaining assertions of
\cref{ass:fixed-pole-orders-dm-family} now follow from the standard
properties of \(\mathcal U_P\) recorded above and from
\cref{prop:toric-adapted-stack-compactification}: the base is smooth
and connected, the compactified family is a proper smooth product
with projective coarse fibres, and its labelled boundary is relative
SNC.
\end{proof}

\begin{corollary}[Coefficient bundles for toric Newton families]
\label{cor:toric-newton-coefficient-bundles}
All conclusions of
\cref{thm:irregular-hodge-bundles-fixed-pole-orders} hold for the
universal potential on \(\mathscr U_\Gamma\times\mathcal U_P\).  In
particular, its twisted de Rham groups, canonical irregular Hodge
levels, and rational graded pieces form vector bundles,
commute with arbitrary coefficient base change, and carry the
integrable connection and shifted transversality described there.
This proves the ordinary assertion of
\cref{thm:introduction-coefficient-invariance}.
\end{corollary}

\begin{proof}
Apply \cref{thm:irregular-hodge-bundles-fixed-pole-orders} to
\cref{thm:toric-family-fixed-pole-orders}.
\end{proof}

\subsection{Inertia sectors and orbifold invariance}
\label{subsec:toric-inertia-orbifold-conclusion}

It remains to verify the fixed-pole compactification hypothesis for
the inertia sectors.  Write the
inertia stack as a finite disjoint union of connected components
\[
 I\mathscr U_\Gamma
 =\coprod_{\gamma\in\mathcal B}\mathscr U_\gamma,
 \qquad
 e_\gamma\colon\mathscr U_\gamma\longrightarrow\mathscr U_\Gamma,
\]
and let \(a_\gamma\in\mathbb Q_{\geq0}\) be the age defined by
\eqref{eq:general-relative-sector-age}.  Let
\(\mathscr U_\gamma^{\mathrm{rig}}\) be the rigidification of
\(\mathscr U_\gamma\) by the subgroup of its generic stabilizer that
acts trivially on the entire sector; we call it the \emph{effective
rigidification}.

\begin{proposition}[Toric sector restriction]
\label{prop:toric-sector-restriction}
For every \(\gamma\in\mathcal B\), the following hold.
\begin{enumerate}[label=\textup{(\roman*)}]
\item The sector \(\mathscr U_\gamma\) is a smooth quasiprojective
  toric Deligne--Mumford stack, possibly with a torus factor and a
  finite generic gerbe.  The stack
  \(\mathscr U_\gamma^{\mathrm{rig}}\) is again of this form.
\item There is a cone \(\sigma_\gamma\in\Gamma\) such that the
  character lattice and Newton polytope of the sector are
  \begin{equation}
   M_\gamma:=M\cap\sigma_\gamma^\perp,
   \qquad
   P_\gamma:=P\cap(M_\gamma)_{\mathbb R}.
  \label{eq:toric-sector-newton-datum}
  \end{equation}
  The polytope \(P_\gamma\) is a face of \(P\) containing \(0\),
  and \((\Gamma_\gamma,P_\gamma)\) satisfies the support condition
  \[
   |\Gamma_\gamma|
   \subseteq\operatorname{Cone}(P_\gamma)^\vee.
  \]
\item Coefficient restriction gives a linear map
  \begin{equation}
   \operatorname{res}_\gamma\colon V_P\longrightarrow V_{P_\gamma},
   \qquad
   \sum_ma_mx^m\longmapsto
   \sum_{m\in P_\gamma\cap M}a_mx^m,
  \label{eq:toric-sector-coefficient-restriction}
  \end{equation}
  such that
  \[
   e_\gamma^*h=\operatorname{res}_\gamma(h),
   \qquad
   \operatorname{res}_\gamma(\mathcal U_P)
   \subseteq\mathcal U_{P_\gamma}.
  \]
\item The product family of restricted potentials on
  \(\mathscr U_\gamma\times\mathcal U_P\) admits compactified data
  satisfying
  \cref{ass:sectorwise-fixed-pole-orders}.  The same is true
  for \(\mathscr U_\gamma^{\mathrm{rig}}\).
\end{enumerate}
\end{proposition}

\begin{proof}
Apply the box decomposition of
\cite[Lemma~4.6 and Proposition~4.7]{BCS} to the complete stacky fan
from \cref{prop:toric-adapted-stack-compactification}, and restrict
it to the unchanged open fan.  This proves \textup{(i)} and attaches
to each nonempty connected sector a minimal cone
\(\sigma_\gamma\).  Rigidification removes precisely the subgroup
specified above and does not change the rational fan, the coarse
character lattice, or the tangent representation defining the age.

Choose \(u_\gamma\in\operatorname{relint}(\sigma_\gamma)\), with
\(u_\gamma=0\) for the untwisted sector.  The support condition gives
\(\langle u_\gamma,m\rangle\geq0\) for every \(m\in P\), and equality
holds precisely for \(m\in\sigma_\gamma^\perp\).  Hence
\[
 P_\gamma
 =\{m\in P\mid\langle u_\gamma,m\rangle=0\}
\]
is a face containing \(0\).  A cone of the sector fan comes from a
cone of \(\Gamma\) containing \(\sigma_\gamma\); passing to the
quotient by \(\operatorname{Span}(\sigma_\gamma)\) proves the stated
support condition.

On a toric chart, restriction to the sector sends \(x^m\) to itself
when \(m\in\sigma_\gamma^\perp\) and to zero otherwise.  This proves
the formula for \(e_\gamma^*h\).  Every face of the face
\(P_\gamma\) is a face of \(P\), and every vertex of
\(P_\gamma\) is a vertex of
\(P\).  Hence \(\Delta_\infty(h)=P\) and Newton nondegeneracy at
infinity imply
\[
 \Delta_\infty(e_\gamma^*h)=P_\gamma
\]
and Newton nondegeneracy at infinity for the restricted potential.
This proves \textup{(iii)}.

Finally, apply \cref{thm:toric-family-fixed-pole-orders} to the toric
sector and \(P_\gamma\), and pull the resulting product family back along
\[
 \operatorname{res}_\gamma|_{\mathcal U_P}\colon
 \mathcal U_P\longrightarrow\mathcal U_{P_\gamma}.
\]
This supplies precisely the compactified data required in
\cref{ass:sectorwise-fixed-pole-orders}.  The rational fan and
coarse character lattice are unchanged by effective rigidification,
so the same construction applies to
\(\mathscr U_\gamma^{\mathrm{rig}}\).
\end{proof}

\begin{corollary}[Orbifold coefficient invariance]
\label{cor:toric-orbifold-coefficient-invariance}
For the universal toric Newton family over \(\mathcal U_P\), all
conclusions of
\cref{thm:orbifold-bundles-fixed-pole-orders} hold.  Thus, for every
\(r,\lambda\in\mathbb Q\), the age-shifted orbifold twisted de Rham
groups, their canonical filtered levels, and their rational graded
pieces form vector bundles over \(\mathcal U_P\), commute with
arbitrary coefficient base change, and carry the integrable
connection with shifted transversality.  Their ranks, and hence the
orbifold irregular Hodge numbers, are independent of the coefficients
in \(\mathcal U_P\).  The same conclusions hold sectorwise after
effective rigidification.  This proves the orbifold assertion of
\cref{thm:introduction-coefficient-invariance}.
\end{corollary}

\begin{proof}
Combine
\cref{thm:toric-family-fixed-pole-orders,prop:toric-sector-restriction}
with \cref{thm:orbifold-bundles-fixed-pole-orders}.
\end{proof}

\subsection{The Harder--Lee combinatorial formula}
\label{subsec:toric-harder-lee-formula}

We use Harder--Lee's notation without repeating its construction.  For
an ordered Clarke dual pair
\((\boldsymbol\Sigma,\check{\boldsymbol\Sigma})\), let
\(\Xi^{\lambda,\mu}\) be the bigraded cellular sheaves on
\((\boldsymbol\Sigma\oplus\check{\boldsymbol\Sigma})_0\), and let
\[
 H^{\lambda,\mu}
 \bigl((\boldsymbol\Sigma\oplus
        \check{\boldsymbol\Sigma})_0\bigr)
 :=\bigoplus_j H^j(\Xi^{\lambda,\mu-j})
\]
be the associated combinatorial Hodge spaces of
\cite[Definition~5.6 and equation~(13)]{HarderLee}.  Write
\(\Xi_{\mathrm{id}}^{\lambda,\mu}\) for the summand in which the
first box-space factor \(B_{c_1}\), corresponding to the inertia of
\(T(\boldsymbol\Sigma)\), is generated by the zero box element; the
second factor \(B_{c_2}\), carrying the Newton grading, is retained.
Set
\begin{equation}
 H_{\mathrm{id}}^{\lambda,\mu}
 \bigl((\boldsymbol\Sigma\oplus
        \check{\boldsymbol\Sigma})_0\bigr)
 :=\bigoplus_j
H^j(\Xi_{\mathrm{id}}^{\lambda,\mu-j}).
\label{eq:toric-harder-lee-identity-combinatorial-space}
\end{equation}

\begin{lemma}[Identity-sector summand]
\label{lem:toric-harder-lee-identity-summand}
The spaces \(\Xi_{\mathrm{id}}^{\lambda,\mu}\) form a canonical
bigraded direct-summand cellular sheaf of
\(\Xi^{\lambda,\mu}\).  Under the \v{C}ech identification in the
proof of \cite[Theorem~5.14]{HarderLee}, its cohomology is the
ordinary tropical Hodge space of
\(T(\boldsymbol\Sigma,
w(\check{\boldsymbol\Sigma})_\varphi)_0\).
\end{lemma}

\begin{proof}
For a stacky cone \(c\), let \(\gamma(g)\) denote the dimension of
the minimal cone containing a box element \(g\).  The decomposition
\[
 B_c
 =\mathbb C[0_c]
  \oplus
  \bigoplus_{\gamma(g)>0}\mathbb C[g]
\]
is preserved by every face projection \(p(c,c')\); this is precisely
the additional grading used in the proof of
\cite[Proposition~4.17]{HarderLee}.  The restriction maps in
\cite[Definition~5.6]{HarderLee} are
\(\varpi\otimes p(c_1,c'_1)\otimes p(c_2,c'_2)\), so projection of
the first box-space factor onto \(\mathbb C[0_{c_1}]\) commutes with
all restrictions.  This proves the direct-summand assertion.

For the model with ambient fan \(\boldsymbol\Sigma\), the first box
factor records its inertia sector, whereas the second is the
box-space factor in the ordinary tropical Jacobian sheaf.  On the
zero summand the former has age zero and contributes one copy of
\(\mathbb C\).  Comparing \cite[Proposition~4.15]{HarderLee} with
the local identification in equation~(15) of that paper therefore
identifies the resulting \v{C}ech complex with the ordinary tropical
Jacobian complex.  The asserted cohomological identification follows
by the proof of \cite[Theorem~5.14]{HarderLee}.
\end{proof}

\begin{corollary}[Harder--Lee formula from a fixed Newton datum]
\label{cor:toric-harder-lee-all-nondegenerate-coefficients}
Assume that \(Q=0\), that \(\mathbf L=L\) is free, and that \(P\) is
full-dimensional.  Suppose moreover that
\begin{equation}
 \Delta_\Gamma^{\mathrm{st}}
 :=
 \bigcup_{\sigma\in\Gamma}
 \operatorname{Conv}\!\left(
  \{0\}\cup
  \{\beta(e_\rho)\mid\rho\in\sigma(1)\}
 \right)
\label{eq:toric-harder-lee-gamma-polytope}
\end{equation}
is convex.  Define \(\boldsymbol\Sigma\) to have underlying fan
\(\Gamma\) and stacky ray vectors \(\beta(e_\rho)\).  Choose a
coherent star triangulation \(\mathscr T\) of the point configuration
\[
 A_P:=\{0\}\cup\operatorname{Vert}(P)
\]
based at \(0\), for example a coherent pulling triangulation with
\(0\) pulled first.  Define \(\check{\boldsymbol\Sigma}\) from the
fan
\begin{equation}
 \check\Sigma_{\mathscr T}
 :=
 \left\{
  \operatorname{Cone}(\tau\setminus\{0\})
  \mathrel{}\middle|\mathrel{}
 \tau\in\mathscr T,\ 0\in\tau
 \right\},
\label{eq:toric-harder-lee-check-sigma-construction}
\end{equation}
by taking each nonzero vertex
\(m\in\operatorname{Vert}(P)\setminus\{0\}\) as the stacky vector on
the ray \(\mathbb R_{\geq0}m\).
Then, for every \(h\in\mathcal U_P\), \(n\in\mathbb Z\), and
\(\lambda\in\mathbb Q\),
\begin{equation}
 \dim_{\mathbb C}
 \Gr_{F_{\mathrm{irr}}}^{\lambda}
 H_{\mathrm{dR}}^n(\mathscr U_\Gamma,h)
 =
 \dim_{\mathbb C}
 H_{\mathrm{id}}^{\lambda,n-\lambda}
 \bigl((\boldsymbol\Sigma\oplus
        \check{\boldsymbol\Sigma})_0\bigr).
\label{eq:toric-harder-lee-all-coefficients}
\end{equation}
In Harder--Lee's orbifold convention, for every
\(r,\lambda\in\mathbb Q\), one likewise has
\begin{equation}
 \dim_{\mathbb C}
 \Gr_{F_{\mathrm{orb}}}^{\lambda}
 H_{\mathrm{orb,dR}}^r(\mathscr U_\Gamma,h)
 =
 \dim_{\mathbb C}
 H^{\lambda,r-\lambda}
 \bigl((\boldsymbol\Sigma\oplus
        \check{\boldsymbol\Sigma})_0\bigr).
\label{eq:toric-harder-lee-orbifold-all-coefficients}
\end{equation}
\end{corollary}

\begin{proof}
The standard pulling construction with \(0\) first gives a coherent
star triangulation and uses every vertex of \(A_P\).  The cones in
\eqref{eq:toric-harder-lee-check-sigma-construction}, together with their
faces, form a simplicial fan with support
\(\operatorname{Cone}(P)\).  Since \(P\) is full-dimensional, its
stacky ray vectors span \(M_\mathbb Q\); hence
\(\check{\boldsymbol\Sigma}\) is a stacky fan in the sense of
\cite[\S2.2]{HarderLee}.  In particular,
\(\check{\boldsymbol\Sigma}\) is quasiprojective, and
\[
 \Delta_{\check{\boldsymbol\Sigma}}
 :=
 \bigcup_{c\in\check{\boldsymbol\Sigma}}
 \operatorname{Conv}\bigl(\{0\}\cup c[1]\bigr)
 =P.
\]
Thus \(\check{\boldsymbol\Sigma}\) is convex in the sense of
\cite[Definition~5.1]{HarderLee}.  The fan
\(\boldsymbol\Sigma\) is simplicial and quasiprojective by
the standing assumptions on \(\Gamma\), and its convexity is exactly
the hypothesis that
\eqref{eq:toric-harder-lee-gamma-polytope} is convex.

The support condition gives
\(\langle n,m\rangle\geq0\) for
\(n\in|\Gamma|\) and \(m\in P\).  Since
\[
 |\boldsymbol\Sigma|=|\Gamma|,
 \qquad
 |\check{\boldsymbol\Sigma}|=\operatorname{Cone}(P),
\]
the pair
\((\boldsymbol\Sigma,\check{\boldsymbol\Sigma})\) is Clarke dual in
the sense of \cite[Definition~5.2]{HarderLee}.  Equality of the
underlying fans and stacky ray vectors identifies the irrelevant
ideals and torus homomorphisms in the Cox quotient presentations,
giving
\[
 T(\boldsymbol\Sigma)\simeq\mathscr U_\Gamma.
\]

Put \(A=\{0\}\cup\check{\boldsymbol\Sigma}[1]\).  For every face
\(F\preceq P\), the set \(A\cap F\) contains the vertices of \(F\)
and hence its differences span a finite-index sublattice of \(M_F\).
The induced morphism from \(T_F\) to the torus determined by these
differences is therefore finite \'{e}tale.  Bertini's theorem applied
after this morphism shows that the coefficients for which the face
hypersurface is smooth form a nonempty Zariski-open subset of the
corresponding coefficient torus.  Since \(P\) has only finitely many
faces, the locus
\[
 \mathcal V_A
 :=\left\{(a_m)_{m\in\check{\boldsymbol\Sigma}[1]}
   \in(\mathbb C^\times)^{\check{\boldsymbol\Sigma}[1]}
   \mathrel{}\middle|\mathrel{}
   1+\sum_m a_mx^m\in\mathcal U_P\right\}
\]
is nonempty and Zariski open.

Coherence supplies, after multiplying by a positive integer, an
integral strictly convex \(\check{\boldsymbol\Sigma}\)-linear support
function \(\varphi\) inducing \(\mathscr T\).  Choose the
Harder--Lee coefficients \(u_m\) generically enough that
\((u_m)_m\in\mathcal V_A\).  The inverse image of the closed set
\((\mathbb C^\times)^{\check{\boldsymbol\Sigma}[1]}\setminus\mathcal V_A\)
under
\[
 \mathbb C^\times\longrightarrow
 (\mathbb C^\times)^{\check{\boldsymbol\Sigma}[1]},
 \qquad
 t\longmapsto
 \bigl(u_mt^{\varphi(m)}\bigr)_m,
\]
is a proper closed subset of \(\mathbb C^\times\), and is therefore
finite.  After shrinking the punctured disc, it contains no
\(\varepsilon\) with \(0<|\varepsilon|\ll1\).  For such
\(\varepsilon\), set
\[
 h_\varepsilon
 :=w(\check{\boldsymbol\Sigma})_{\varphi,\varepsilon}.
\]
All coefficients of
\(h_\varepsilon\) indexed by
\(\check{\boldsymbol\Sigma}[1]\) are nonzero, so
\(\Delta_\infty(h_\varepsilon)=P\); the definition of
\(\mathcal V_A\) then gives
\(h_\varepsilon\in\mathcal U_P\).

This face condition agrees with the
nondegeneracy convention of \cite[\S2.1]{HarderLee}.  If
\(0\notin F\), choose \(v_F\in N_\mathbb Q\) and
\(a_F\in\mathbb Q^\times\) such that
\(\langle v_F,m\rangle=a_F\) on \(F\).  The identity
\[
 \partial_{v_F}(h_\varepsilon)_F
 =a_F(h_\varepsilon)_F
\]
identifies their logarithmic critical-point criterion with the
smoothness of
\(Z\bigl(\overline{(h_\varepsilon)}_F\bigr)\subset T_F\).

The family
\[
 \bigl(T(\boldsymbol\Sigma),
       w(\check{\boldsymbol\Sigma})_\varphi\bigr)
\]
therefore satisfies the hypotheses of
\cite[Theorem~5.15]{HarderLee}, and its fibres with
\(0<|\varepsilon|\ll1\) lie in \(\mathcal U_P\).  Fix one such
\(\varepsilon\).
For \(h=h_\varepsilon\), \cite[Theorem~5.15]{HarderLee}, proved as
\cite[Corollary~9.22]{HarderLee}, expresses the irregular Hodge
numbers in terms of the ordinary tropical Hodge spaces of the central
fibre.  By \cref{lem:toric-harder-lee-identity-summand}, these are
\(H_{\mathrm{id}}^{\lambda,n-\lambda}
((\boldsymbol\Sigma\oplus\check{\boldsymbol\Sigma})_0)\).  This proves
\eqref{eq:toric-harder-lee-all-coefficients} for \(h_\varepsilon\).
By \cref{cor:toric-newton-coefficient-bundles}, the dimension of its
left-hand side is constant on the connected coefficient space
\(\mathcal U_P\).  This proves the equation for every
\(h\in\mathcal U_P\).

For \(h_\varepsilon\), the orbifold degeneration formula
\cite[equation~(16)]{HarderLee}---which uses the orbifold tropical
identification of \cite[Proposition~4.17]{HarderLee} and the extension
recorded in \cite[Remark~9.23]{HarderLee}---computes the orbifold
irregular Hodge numbers by the full orbifold tropical Jacobian
complex.  The \v{C}ech comparison in the proof of
\cite[Theorem~5.14]{HarderLee} identifies this complex with
\(\Xi^{\lambda,\mu}\), including both box-space factors.  This proves
\eqref{eq:toric-harder-lee-orbifold-all-coefficients} for
\(h_\varepsilon\); its extension to every \(h\in\mathcal U_P\) follows
from
\cref{cor:toric-orbifold-coefficient-invariance}.
\end{proof}

\bibliographystyle{amsalpha}
\providecommand{\bysame}{\leavevmode\hbox to3em{\hrulefill}\thinspace}
\providecommand{\MR}{\relax\ifhmode\unskip\space\fi MR }
\providecommand{\MRhref}[2]{\href{http://www.ams.org/mathscinet-getitem?mr=#1}{#2}
}
\providecommand{\href}[2]{#2}

\end{document}